\documentclass{article}

\usepackage{microtype}
\usepackage{graphicx}
\usepackage{subfigure}
\usepackage{booktabs} 
\usepackage{amscd}
\usepackage{amsthm}
\usepackage{eucal}
\usepackage{tikz-cd}
\usepackage{hyperref}

\usepackage[accepted]{icml2023}

\usepackage{amsmath}
\usepackage{amssymb}
\usepackage{mathtools}
\usepackage{amsthm}
\usepackage{color}

\usepackage[capitalize,noabbrev]{cleveref}

\theoremstyle{plain}
\newtheorem{theorem}{Theorem}[section]

\theoremstyle{definition}
\newtheorem{definition}[theorem]{Definition}

\theoremstyle{remark}
\newtheorem{remark}[theorem]{Remark}
\newtheorem{example}[theorem]{Example}

\newcommand{\R}{\mathbb{R}}
\newcommand{\cat}[1]{\mathrm{\textbf{#1}}}

\makeatletter
\newcommand{\colim@}[2]{%
  \vtop{\m@th\ialign{##\cr
    \hfil$#1\operator@font colim$\hfil\cr
    \noalign{\nointerlineskip\kern1.5\ex@}#2\cr
    \noalign{\nointerlineskip\kern-\ex@}\cr}}%
}
\newcommand{\colim}{%
  \mathop{\mathpalette\colim@{\rightarrowfill@\textstyle}}\nmlimits@
}
\makeatother

\icmltitlerunning{Topologically Attributed Graphs for Shape Discrimination}

\begin{document}

\twocolumn[
\icmltitle{Topologically Attributed Graphs for Shape Discrimination}



\icmlsetsymbol{equal}{*}

\begin{icmlauthorlist}
\icmlauthor{Justin Curry}{equal,aaa}
\icmlauthor{Washington Mio}{equal,bbb}
\icmlauthor{Tom Needham}{equal,bbb}
\icmlauthor{Osman Berat Okutan}{equal,ccc}
\icmlauthor{Florian Russold}{equal,ddd}
\end{icmlauthorlist}

\icmlaffiliation{aaa}{Department of Mathematics, SUNY Albany, Albany, NY, USA}
\icmlaffiliation{bbb}{Department of Mathematics, Florida State University, Tallahassee, FL, USA}
\icmlaffiliation{ccc}{Max Planck Institute for Mathematics in the Sciences, Leipzig, Germany}
\icmlaffiliation{ddd}{Institute of Geometry, Graz University of Technology, Graz, Austria}

\icmlcorrespondingauthor{Tom Needham}{tneedham@fsu.edu}

\icmlkeywords{Topological Data Analysis, Reeb Graphs, Persistent Homology, Graph Neural Networks, Gromov-Hausdorff Distance}

\vskip 0.3in
]



\printAffiliationsAndNotice{\icmlEqualContribution} 

\begin{abstract}
In this paper we introduce a novel family of attributed graphs for the purpose of shape discrimination.
Our graphs typically arise from variations on the Mapper graph construction, which is an approximation of the Reeb graph for point cloud data.
Our attributions enrich these constructions with (persistent) homology in ways that are provably stable, thereby recording extra topological information that is typically lost in these graph constructions. We provide experiments which illustrate the use of these invariants for shape representation and classification. In particular, we obtain competitive shape classification results when using our topologically attributed graphs as inputs to a simple graph neural network classifier.
\end{abstract}

\section{Introduction}

Topological Data Analysis studies finite spaces by associating topological invariants to them that serve as intuitive structural summaries for unsupervised analysis or as nonlinear featurizations for downstream supervised learning applications. The most common such invariant is a persistence diagram, which, roughly, gives a concise representation of homological features that are apparent in the data at multiple scales. Graphical topological summaries form another important collection of tools for representing data; these include merge trees,  Mapper graphs, and, their continuous counterparts, Reeb graphs. 

In this paper, we combine Mapper graphs with persistence diagrams in order to define new, highly discriminative shape representations. The main ideas of these constructions are illustrated in Figures \ref{fig:torus-DRG} and \ref{fig:intro_example}. Roughly, the Mapper graph gives a large-scale structural summary of connected components, while the persistence diagram attributions encode finer-scale topological structure.

The structure of the paper is as follows. In Sections \ref{sec:attributed-graphs} and \ref{sec:DRGs}, we introduce precise mathematical formalism for attributed graphs and their continuous analogues---decorated Reeb graphs---in the language of category theory. We then introduce novel constructions of topologically attributed graphs and prove their stability in Section \ref{sec:persistent_decorations}. Sections \ref{sec:computation} and \ref{sec:examples} are devoted to computational considerations; in particular, how our topologically attributed graphs are constructed and compared in practice. We also provide a classification experiment, where we show that our constructions achieve competitive shape classification performance when they are fed as inputs into a simple graph neural network.

\begin{figure}
    \centering
    \includegraphics[width = 0.3\textwidth]{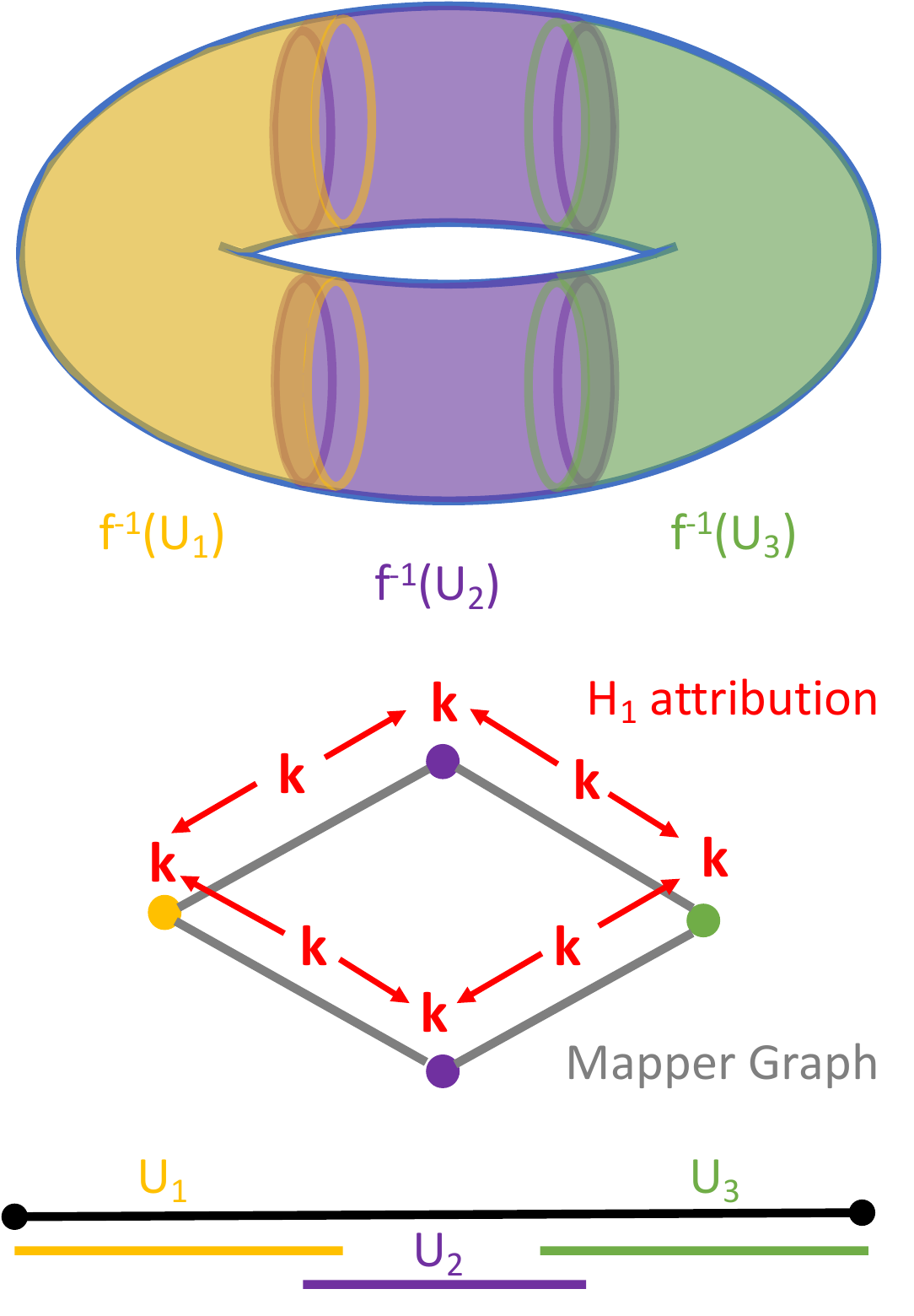}
    \caption{The Mapper graph construction applied to the torus admits a natural decoration with homology (with coefficients in a field $k$).}
    \label{fig:torus-DRG}
\end{figure}

\section{Categorically Attributed Graphs}\label{sec:attributed-graphs}

We view a simple undirected graph $G=(V,E)$  as a category\footnote{See \cite{riehl2017category} for a good introduction to category theory.} $\cat{G}$, with objects corresponding to elements of $V\cup E$ and with a unique morphism $e \to v$ whenever a node $v$ is incident to an edge $e$.
This makes $\cat{G}$ equivalent to a poset $(G,\leq)$ where $e \leq v$.

\begin{definition}[Attributed Graph]
An \emph{attributed graph} is a functor $F\colon \mathbf{G} \rightarrow \mathbf{C}$ which assigns to each vertex $v$ and each edge $e$ of $G$ objects $F(v)$ and $F(e)$ in $\mathbf{C}$, along with a morphism $F(e \leq v): F(e) \to F(v)$ in $\cat{C}$, provided $e\leq v$.
\end{definition}

\begin{example}
    Suppose we are considering a social media platform, such as Facebook.
    Users correspond to nodes of a graph $G$ and edges correspond to friendships.
    We can define an attribution valued in the category $\cat{Set}$ as follows: Let $F(v)$ to be the set of interests or pages that $v$ follows and let $F(e)$ be the intersection of these interests or pages. The inclusion $F(e)\hookrightarrow F(v)$ makes this an attribution.
\end{example}

\paragraph{Attributed Graphs for Representing Shapes.} In this paper, we are primarily interested in attributed graphs which capture aspects of the geometry and/or topology of a given space (or finite approximation thereof). As such, we will mostly work with attributions that come from homology\footnote{See \cite{hatcher2002algebraic} for a textbook treatment.}, which is an attribution valued in $\cat{Vec}$---the category of vector spaces and linear maps over a field $k$. A first example of the type of attributed graph we are interested in is as follows.

\begin{example}[Decorated Mapper Graphs]\label{ex:DMG}
Let $X$ be a compact space and assume $f\colon X \rightarrow \mathbb{R}$ is a continuous map. 
Let $\mathcal{U}=(U_i)_{i\in I}$ be a cover of $\mathbb{R}$ with no (non-empty) triple intersections. 
We can pullback this cover along $f$ to obtain  $f^{-1}(\mathcal{U})$ as a cover of $X$, where each cover element $f^{-1}(U_i)$ is further refined into its connected components. 
The nerve of this cover defines the \emph{Mapper graph}  $\mathcal{M}_{\mathcal{U},f}$ of $X\xrightarrow{f} \mathbb{R}$ with respect to $\mathcal{U}$ \cite{singh2007topological}. 
It has vertices $V$ corresponding to components $C$ of $f^{-1}(U_i)$ and edges $E=\{C\cap C'\hspace{2pt}|\hspace{2pt} C,C'\in f^{-1}(\mathcal{U}) \text{ and }C\cap C'\neq\emptyset\}$ corresponding to non-empty intersections of these components. 
The \emph{decorated Mapper graph (DMG)} $F\colon \mathcal{M}_{\mathcal{U},f}\rightarrow \mathbf{Vec}$ augments the Mapper graph by assigning to each component $C\in V$ and $C\cap C'\in E$ the homology (with coefficients in a field $k$) of the corresponding components, i.e.~$F(C)\coloneqq H_n(C)$ and $F(C\cap C')\coloneqq H_n(C\cap C')$.
The inclusion $C\cap C'\subseteq C$ of components induces a map in homology $F(C\cap C' \leq C)\coloneqq H_n(C\cap C'\subseteq C)$. 

An example of a DMG is shown in Figure \ref{fig:torus-DRG}.
\end{example}

\paragraph{Discrete Shape Representations and TDA.} In our next example, we consider an extension of the DMG concept which applies to finite spaces.
As a reminder, Topological Data Analysis (TDA) provides a tool for homology inference that replaces homology with \emph{persistent homology}; we assume that the reader is familiar with the basic concepts of TDA, but review one key construction below.

\begin{figure}
    \centering
    \includegraphics[width = 0.48\textwidth]{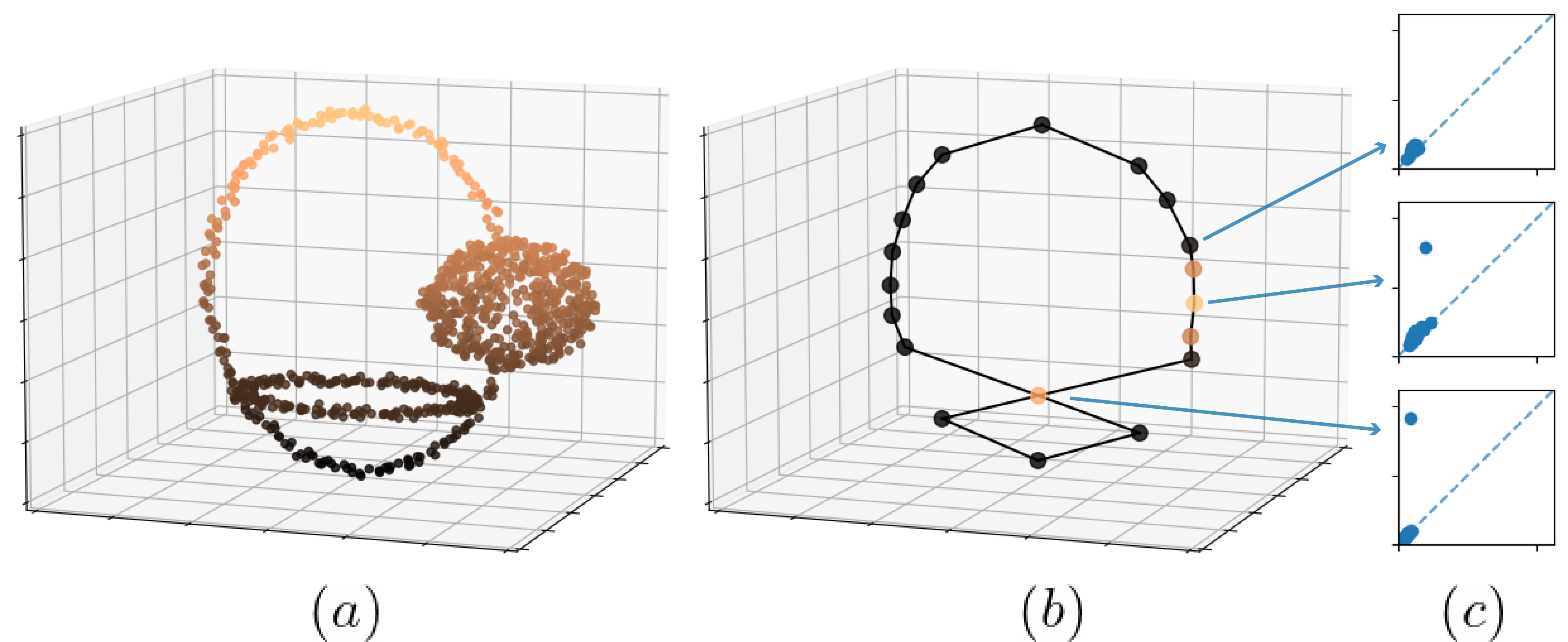}
    \caption{Persistent Decorated Mapper Graph. (a) A synthetic point cloud data set, nodes colored by the value of its filtration function, height along the $z$-axis. (b) A Mapper graph of the dataset. (c) The nodes of the Mapper graph are attributed with persistence diagrams; each node corresponds to a connected component of a level set of the dataset, and the (degree-1) persistent homology of this subset gives the attribution. Nodes of the Mapper graph are colored by total persistence (i.e., $\sum (d_\mathrm{I} - b_i)$, where the sum is over points $(b_i,d_\mathrm{I})$ in the diagram).}
    \label{fig:intro_example}
\end{figure}

\begin{definition}[Rips Persistence]\label{defn:rips}
    Given a finite metric space $(X,d_X)$ the \emph{Vietoris-Rips complex at scale $r$} is the simplicial complex $VR(X,r)$ whose simplices consist of subsets $\sigma\subseteq X$ where $d_X(x,x')\leq 2r$ for all $x,x'\in\sigma$.
    Notice that if $r\leq s$, then there is an inclusion $VR(X,r)\subseteq VR(X,s)$.
    Passing to the geometric realization of these complexes and the induced continuous maps makes $VR(X):=VR(X,\bullet)$ into a functor from $(\R,\leq)$---the poset category of the reals---to $\cat{Top}$---the category of topological spaces and maps, i.e. $VR(X)$ is an object in the functor category $\cat{Top}^{\R}$.
    Applying homology $H_n$ then defines the Rips persistent homology $PH_n(X)$, which is an object in $\cat{Vec}^{\mathbb{R}}$.
    This latter object can then be faithfully encoded as a persistence diagram, which records births and deaths of homological features across scales.
\end{definition}

\begin{definition}[Persistent DMGs]\label{def:persistentDMG}
Given a finite metric space $(X,d_X)$ and function $f:X \to \R$, we can construct the (discrete version of the) Mapper graph $\mathcal{M}_{\mathcal{U},f}$ in a manner similar to Example \ref{ex:DMG} by inferring components via a chosen clustering algorithm applied to $f^{-1}(U_i)$. 
The clusters then replace the components $C\in f^{-1}(U_i)$ in the construction above. 
This allows us to define a \emph{persistent Decorated Mapper Graph} $F\colon \mathcal{M}_{\mathcal{U},f}\rightarrow \mathbf{Vec}^\mathbb{R}$ that assigns to each vertex and each edge---corresponding to a cluster and an intersection of clusters, respectively---the persistent homology of each, i.e.~$F(C)\coloneqq PH_n(C)$ and $F(C\cap C' \leq C)\coloneqq PH_n(C\cap C'\subseteq C)$.

This structure is illustrated in Figure \ref{fig:intro_example}. In this figure, only the node attributions are included.
\end{definition}

Persistent DMGs give intuitive and informative summaries of discrete shapes, and are the main object that we use in applications below (see Sections \ref{sec:computation} and \ref{sec:examples}). The next part of the paper is concerned with establishing basic theory for these objects, focusing on stability.

\section{Decorated Reeb Graphs}\label{sec:DRGs}

In this section, we introduce continuous versions of the Mapper graphs and attributions described above. 

\paragraph{Reeb Graphs.} In practice, choosing the cover and clustering schema for Mapper can be an art with sometimes hard to interpret and unstable behavior.
These defects are then inherited by the decoration process.
These issues have been mostly handled \cite{munch2016convergence,carriere2018structure} by viewing Mapper graphs as discrete approximations of the Reeb graph~\cite{reeb1946points}, which we review next.

\begin{definition}\label{def:Reeb_graph}
    A \emph{Reeb graph} is a pair $(R,f)$ consisting of a compact 1-dimensional geometric simplicial complex $R$ and a piecewise linear  map $f:R \to \R$. 
    A metric $d_f$ on $R$ is defined by
    $d_f(x,x') = \inf_\gamma \max f \circ \gamma - \min f \circ \gamma$,
    where the infimum is over all paths from $x$ to $x'$. 
\end{definition}

\begin{example}\label{ex:Reeb_graph}
    Let $X$ be a compact geometric simplicial complex and let $f:X \to \R$ be a continuous piecewise linear map. 
    The Reeb graph associated to $f:X\to \R$ starts by defining $R$ to be set of equivalence classes $X/\sim$, where $x \sim x'$ if $x$ and $x'$ lie in the same connected component of $f^{-1}(v)$. 
    Since $f$ is constant on equivalence classes, it factors to define a map $\hat{f}:R \to \R$ where $f=\hat{f}\circ q$ and $q$ is the quotient map $q:X\to X/\sim$.
    The pair $(R,\hat{f})$ then defines a Reeb graph in the sense of Definition \ref{def:Reeb_graph}.
\end{example}

We now make geometric graphs the domain of attribution.

\begin{definition}\label{def:cts-attribution}
    Let $R$ be a compact 1D geometric complex and let $\mathcal{O}(R)$ be its poset category of open sets.
    A \emph{continuous attribution} is a functor $F:\mathcal{O}(R) \to \cat{C}$.
\end{definition}

\begin{example}[Decorated Reeb Graph (DRG)]\label{ex:DRG}
    When $(R,f)$ is a Reeb graph and $\cat{C} = \cat{Vec}$, we refer to a continuous attribution $F:\mathcal{O}(R) \to \cat{Vec}$ as a \emph{decorated Reeb graph} or \emph{DRG}.
    Specifically, let $(R,\hat{f})$ be a Reeb graph arising from the construction of Example \ref{ex:Reeb_graph}. 
    The \emph{homology decorated Reeb graph} is the continuously attributed graph that assigns to each open set $U \subset R$ in the Reeb graph the homology $F(U) = H_{n}(q^{-1}(U))$.
\end{example}

\paragraph{Categorical Reeb Graphs.} In order to prove that the Reeb graph is stable, \cite{desilva} used the following definition of a Reeb graph.

\begin{definition}\label{def:cat-Reeb-graph}
    A \emph{categorical Reeb graph} is a functor $\mathcal{R}:\mathcal{O}(\R) \to \cat{Set}$, where $\mathcal{O}(\R)$ is the category of open subsets of $\R$ ordered by inclusion, that satisfies \emph{constructibility}---there exists some finite collection of critical values $\tau=\{t_0,\ldots,t_n\}\subset \R$ such that if $I\subseteq J$ are two intervals with equal intersection with $\tau$, then the map $\mathcal{R}(I\subseteq J)$ is an isomorphism---and the
    \emph{cosheaf axiom}---if $\mathcal{U}=\{U_i\}_{i\in I}$ is a cover of an open set $U\in \mathcal{O}(\R)$ then the universal map from the colimit $\varinjlim \mathcal{R}(U_i)\to \mathcal{R}(U)$ is an isomorphism.
\end{definition}

\begin{example}
    Every Reeb graph $(R,f)$ gives rise to a categorical Reeb graph $\mathcal{R}$ via $\mathcal{R}(U) := \pi_0(f^{-1}(U))$, where
    $\pi_0: \cat{Top} \to \cat{Set}$ is the path components functor.
\end{example}

We now unify Definitions \ref{def:cts-attribution} and \ref{def:cat-Reeb-graph} to provide an alternative description of Example \ref{ex:DRG}.
This involves engineering a category that can track both components and homology vector spaces.

\begin{definition}
    Let $\cat{PVec}$ denote the category of \emph{discretely parameterized vector spaces}. 
    Objects of $\cat{PVec}$ are functors $\sigma:S \to \cat{Vec}$, where $S$ is a set regarded as a discrete category and
    a morphism from $\sigma:S \to \cat{Vec}$ to $\tau:T \to \cat{Vec}$ consists of a set map $\mu:S \to T$ and a natural transformation $\sigma \Rightarrow \tau \circ \mu$.
    This category has a functor $\mathbf{dom}:\cat{PVec}\to \cat{Set}$ that sends $\sigma:S\to \cat{Vec}$ to $S$.
\end{definition}

\begin{definition}\label{def:cat-DRG}
    A \emph{categorical decorated Reeb graph} is a functor $\mathcal{F}:\mathcal{O}(\R) \to \cat{PVec}$ such that $\mathbf{dom} \circ \mathcal{F}$ satisfies the axioms of Definition \ref{def:cat-Reeb-graph}.
\end{definition}

\begin{example}
    Let $F:\mathcal{O}(R) \to \cat{Vec}$ be a DRG. 
    This gives rise to a categorical DRG $\mathcal{F}:\mathcal{O}(\R) \to \cat{PVec}$ where $\mathcal{F}(U)$ is the object of $\cat{PVec}$ that maps $\pi_0(\hat{f}^{-1}(U)) \to \cat{Vec}$ by taking a connected component of $\hat{f}^{-1}(U) \ni A \subset R$ to $F(A)$. 
\end{example}

\section{Persistent Decorations and Stability}\label{sec:persistent_decorations}

One of the main contributions of TDA has been the observation that connected components and homology are stable only when considered as part of a family of topological spaces. We now review the concepts used to quantify this.

\paragraph{Metrics.} Two of the most prominent distance metrics used in TDA are Gromov-Hausdorff distance and interleaving distance, which we now define.

\begin{definition}[Gromov-Hausdorff Distance]
    Let $(X,d_X)$ and $(Y,d_Y)$ be metric spaces. 
    The \emph{distortion} of a pair of (not necessarily continuous) maps $\Phi:X\to Y$ and $\Psi:Y\to X$ is the quantity $\text{dist}(\Phi,\Psi)$ defined by
    \begin{equation*}
        \text{sup} \{|d_X(x,x')-d_Y(y,y')| \mid (x,y),(x',y')\in C(\Phi,\Psi)\},
    \end{equation*}
    where
    \begin{equation*}
        C(\Phi,\Psi)\coloneqq \{(x,y)\in X\times Y \mid y=\Phi(x) \mbox{ or } x=\Psi(y)\}.
    \end{equation*}
    The \emph{Gromov-Hausdorff distance} between $X$ and $Y$ is
    \begin{equation*}
        d_{\mathrm{GH}}(X,Y) \coloneqq \inf_{\Phi,\Psi} \frac{1}{2} \text{dist}(\Phi,\Psi).
    \end{equation*}
\end{definition}

The stability results we are interested in are based on the interleaving construction of TDA.

\begin{definition}[Interleaving Distance] \label{interleaving_distance}
    Let $\mathcal{P}$ be a poset, $\cat{C}$ a category and $\cat{C}^{\mathcal{P}}$ the functor category equipped with a notion of shifting/smoothing
    for any $\epsilon\geq 0$, i.e.~$(\bullet)^{\epsilon}:\cat{C}^{\mathcal{P}}\to\cat{C}^{\mathcal{P}}$, 
    is a functor that sends $F\mapsto F^{\epsilon}$ 
    and this functor is equipped with a natural transformation $\eta^{\epsilon}:\text{id}_{\cat{C}^{\mathcal{P}}}\Rightarrow (\bullet)^{\epsilon}$ that interacts in compatible ways\footnote{This is called a flow structure on the category $\cat{C}^{\mathcal{P}}$, whose full details are explored in \cite{stefanou2018dynamics,de2018theory}.}.
    We say that two objects $F,G\in \cat{C}^{\mathcal{P}}$ are $\epsilon$-\emph{interleaved} if there exist morphisms  
    $\phi:F\to G^{\epsilon}$ and $\psi:G\to F^{\epsilon}$ such that 
    $\eta^{2\epsilon}_F = \psi^{\epsilon}\circ \phi$ and $\eta^{2\epsilon}_G=\phi^{\epsilon}\circ \psi$.
    The \emph{interleaving distance } between $F$ and $G$ is then defined as
    \[
    d_\mathrm{I}(F,G)=\inf \{\epsilon \mid \mbox{$F$ and $G$ are $\epsilon$-interleaved}\}.
    \]
\end{definition}

\begin{example}[Rips Persistence]
    If we choose $\mathcal{P}=(\mathbb{R},\leq)$ and $\cat{C}=\mathbf{Vec}$ in Definition \ref{interleaving_distance} we obtain the usual interleaving distance for 1-parameter persistence modules. Rips persistent homology for a finite metric space $X$ (see Definition \ref{defn:rips}), written $PH_n(X) \in \cat{Vec}^\R$, has a natural notion of shifting by defining $PH_n(X)^{\epsilon}(r)=H_{n}(VR(X)(r+\epsilon))$.
    The fact that $VR(X)$ is a functor provides a map from $VR(X,r)\to VR(X,r+\epsilon)$, which gives the data of the natural transformation $\eta^{\epsilon}$. Thus the interleaving distance between Rips persistent homology functors is well-defined.
\end{example}

Interleavings between Rips persistent homology leads to a foundational stability result of TDA~\cite{chazal2009gromov,chazal2014persistence}: for finite metric spaces $X$ and $Y$ 
\[
d_\mathrm{I}(PH_n(X),PH_n(Y)) \leq d_\mathrm{GH}(X,Y).
\]
The type of stability we're interested in is not only governed by the Gromov-Hausdorff distance between point clouds, but scalar functions on these. 
This is expressed in the following definition, which is equivalent to a metric used in \cite{chazal2009gromov}; see also \cite{bauer2014measuring}, where a similar metric is used in the context of Reeb graphs.

\begin{definition} \label{functional_dist_dist}
Let $X$ and $Y$ be metric spaces equipped with functions $f:X\to\R$ and $g:Y\to\R$, written $X_f$ and $Y_g$, respectively.
If $\Phi:X\to Y$ and $\Psi:Y\to X$ are maps, then the functional distortion of $\Phi$ and $\Psi$ is
\begin{equation*}
\begin{aligned}
\text{FunDist}\big(\Phi,\Psi\big)\coloneqq  \text{ max} \begin{cases} 
\frac{1}{2}\text{dist}(\Phi,\Psi) \\
\lvert\lvert f-g\circ\Phi \rvert\rvert_\infty \\
\lvert\lvert g-f\circ\Psi \rvert\rvert_\infty .
\end{cases}
\end{aligned}
\end{equation*}
The functional distortion distance is then 
\begin{equation*}
d_{\mathrm{FD}}(X_f,Y_g)\coloneqq \inf_{\Phi,\Psi} \text{FunDist}\big(\Phi,\Psi\big).
\end{equation*}
\end{definition}

It is straightforward to show that $d_{\mathrm{FD}}$ is a pseudometric on the space of pairs $(X,f)$. 

\paragraph{Stability of Persistent Discrete DRGs.} We now define a persistent discrete Decorated Reeb Graph construction, which refines the notion of a persistent DMG (Definition \ref{def:persistentDMG}), and will be stable under perturbations of the functional distortion distance.

\begin{definition}[Persistent Discrete DRG]\label{def:persistence_discrete_DRG}
    Given a finite metric space endowed with a scalar-valued function $X_f$, we define  $f_r^{-1}(U)$ to be the full subcomplex of $VR(X,r)$ on all vertices $x \in f^{-1}(U)$, for each open subset $U \subset \R$.
    We then define the \emph{persistent (discretized) decorated Reeb graph} of $X_f$ to be the following 2-parameter family of categorical DRGs (Definition \ref{def:cat-DRG}):
    \begin{equation*}
        DF: (\R^2,\leq) \to \cat{Fun}(\mathcal{O}(\R), \cat{PVec}) \qquad (r,s)\mapsto \mathcal{F}(r,s)
    \end{equation*}
    where
    \begin{equation*}
        \mathcal{F}(r,s)(U):=\{A\in \pi_0(f^{-1}_r(U)) \to H_{\bullet}(VR(A,s))\}.
    \end{equation*}
\end{definition}

\begin{remark}\label{rmk:DRGandDMG}
    For a fixed $r \geq 0$, $\mathcal{F}(r,s)$ assigns to each connected component of $f^{-1}_r(U)$ the Vietoris-Rips persistent homology of that point cloud at scale $s$. Then $DF(r,\bullet)$ can be considered as a DRG $DF(r,\bullet): \mathcal{O}(\R) \to \cat{PVec}^\R$, by setting $DF(r,\bullet)(U)(s) = DF(r,s)(U)$. If we also fix a cover $\mathcal{U}$, we recover the persistent DMG of Definition \ref{def:persistentDMG} by choosing clusters associated to $U \in \mathcal{U}$ to be given by the connected components of $f_r^{-1}(U)$.

    In the above sense, the persistent discrete DRG refines the notion of a persistent DMG. This relaxation to a more continuous and categorical setting is crucial to our proof of the stability result below.
\end{remark}

\begin{theorem}[Stability of Persistent Discrete DRGs]\label{thm:persistent_DRGs}
Let $X_f$ and $Y_g$ be finite metric spaces endowed with scalar-valued functions. Let $DF$ and $DG$ be their respective persistent discrete DRGs (Definition \ref{def:persistence_discrete_DRG}). Then we have 
\[
d_\mathrm{I}(DF,DG) \leq d_\mathrm{FD}(X_f,Y_g).
\]
\end{theorem}

The $\epsilon$-smoothing of $DF$ is defined by
\[
DF^{\epsilon}(r,s)(U)=DF(r+\epsilon,s+\epsilon)(U^{\epsilon})
\]
where $U^{\epsilon}:=\{t\in \R \mid \exists v\in U \text{ s.t. } |t-v|<\epsilon\}$ is the $\epsilon$-thickening of the open set $U\in \mathcal{O}(\R)$. This leads to the notion of the interleaving distance $d_\mathrm{I}$ used in the theorem.

\begin{proof}[Proof Sketch]
Suppose the functional distortion distance of Definition \ref{functional_dist_dist} between $X_f$ and $Y_g$ is less than $\delta$. 
This means that for every $\epsilon>\delta$ there are maps $\Phi:X \to Y$ and $\Psi:Y\to X$ whose distortion is less than $2\epsilon$.
Also, $\lvert\lvert f-g\circ \Phi \rvert\rvert_\infty \leq \epsilon$, which implies that $\forall U \in \mathcal{O}(\R)$ we have
\[
    f^{-1}(U)\subseteq \Phi^{-1}(g^{-1}(U^{\epsilon})).
\]
This implies that if $\sigma\subseteq f^{-1}(U)$ is a subset with $d(x_i,x_j)\leq 2r$ for all pairs of points in $\sigma$, i.e. $\sigma\in VR(X,r)\cap f^{-1}_r(U)$, then $\Phi_{r,s}(\sigma)\in g^{-1}_{r+\epsilon}(U^{\epsilon}) \cap VR(Y,s+\epsilon)$.
Moreover, this containment holds when restricted to a component $A\in \pi_0(f^{-1}_r(U))$.
Symmetric reasoning using the condition $\lvert\lvert g-f\circ \Psi \rvert\rvert_\infty \leq \epsilon$ guarantees that 
$\forall U \in \mathcal{O}(\R)$
\[
    g^{-1}(U)\subseteq \Psi^{-1}(f^{-1}(U^{\epsilon}))
\]
and in particular $\Phi_{r,s}(\sigma)\in g^{-1}_{r+\epsilon}(U^{\epsilon}) \cap VR(Y,s+\epsilon)$ is carried to a simplex $\Psi_{r+\epsilon,s+\epsilon}\circ\Phi(\sigma)\in f^{-1}_{r+2\epsilon}(U^{2\epsilon}) \cap VR(X,s+2\epsilon)$ that is contiguous to $\sigma$ inside $ f^{-1}_{r+2\epsilon}(U^{2\epsilon}) \cap VR(X,s+2\epsilon)$, thus guaranteeing that the induced map on homology $VR(A,s)\to VR(A,s+\epsilon)$ for each component $A\in \pi_0(f_r^{-1}(U))$ is the same as $\Psi_{r+\epsilon,s+\epsilon}\circ\Phi_{r,s}$.
This establishes half of the interleaving condition and the other half is argued \emph{mutatis mutandi}.
\end{proof}

\paragraph{Stability of Barcode Transforms.} We end this theoretical section with another stability result, which deals more directly with DRGs. While the constructions involved are somewhat more straightforward, we discuss their limitations in practice at the end of the section.

\begin{definition}[Barcode Transform]\label{def:barcode-transform}
Let $F\colon \mathcal{O}(R)\rightarrow \mathbf{Vec}$ be the decorated Reeb graph associated to $f:X\to \R$, where each open set $U\subseteq R$ is assigned a finite-dimensional vector space. 
We define the \emph{barcode transform} of $F$ to be the map
\begin{equation*}
\begin{aligned}
BF\colon R&\rightarrow \mathbf{Vec}^{\mathbb{R}} \\
r\in R& \mapsto \Big( t\in\mathbb{R}_{\geq 0}\mapsto F\big( B_{d_f}(r,t)\big) \Big)
\end{aligned}
\end{equation*}
Since every persistence module can be identified with a barcode, we can view the barcode transform as an assignment of a barcode to each point in the Reeb graph. 
\end{definition}

Using $\epsilon$-smoothings of open sets, i.e.\ setting $\mathcal{P}=\mathcal{O}(\mathbb{R})$ and $F^\epsilon(U)\coloneqq F(U^{\epsilon})$ in Definition \ref{interleaving_distance}, we can define interleaving distances for categorical Reeb graphs and categorical decorated Reeb graphs. Moreover, the interleaving distance of categorical Reeb graphs gives rise to an interleaving distance of concrete Reeb graphs as defined in \cite{desilva}. In the following we define the functional distortion distance for barcode transforms and show that it is controlled by the interleaving distance of the Reeb graphs and their corresponding categorical decorated Reeb graphs.   

\begin{definition} \label{def:FD_distance_BT}
    Let $F\colon \mathcal{O}(R)\rightarrow\mathbf{Vec}$ and $G\colon \mathcal{O}(S)\rightarrow\mathbf{Vec}$ be concrete decorated Reeb graphs over $(R,f)$ and $(S,g)$. We define the functional distortion distance of the corresponding barcode transforms by
    \begin{equation*}
        d_\mathrm{FD}\big(BF,BG\big)\coloneqq \underset{\Phi,\Psi}{\text{inf}} \text{max}\begin{cases}
            \text{FunDist}(\Phi,\Psi) \\
            \underset{r\in R}{\text{sup}}\hspace{2pt} d_\mathrm{I}\big(BF(r),BG\circ \Phi(r)\big) \\
            \underset{s\in S}{\text{sup}}\hspace{2pt} d_\mathrm{I}\big(BF\circ\Psi(s),BG(s)\big)
        \end{cases}
    \end{equation*}
    where $\text{FunDist}$ is taken w.r.t.\ $f$ and $g$ and the infimum is over all functions $\Phi\colon R\rightarrow S$ and $\Psi\colon S\rightarrow R$.
\end{definition}

\begin{theorem} \label{thm:strong_equivalence}
    Let  
    $F\colon \mathcal{O}(R)\rightarrow \mathbf{Vec}$ and $G\colon \mathcal{O}(S)\rightarrow \mathbf{Vec}$ be concrete decorated Reeb graphs over $(R,f)$ and $(S,g)$ and $\mathcal{F},\mathcal{G}\colon \mathcal{O}(\mathbb{R})\rightarrow \mathbf{PVec}$ the corresponding categorical decorated Reeb graphs, then
    \begin{equation*}
        d_\mathrm{I}\big(R,S\big)\leq d_\mathrm{FD}\big(BF,BG\big)\leq 6 d_\mathrm{I}\big(\mathcal{F},\mathcal{G}\big) \hspace{1pt}.
    \end{equation*}
\end{theorem}

\begin{proof}[Proof Sketch]
    We begin with the inequality on the left. As shown in \cite{bauer}, $d_\mathrm{I}\big(R,S\big)\leq d_\mathrm{FD}\big(R_f,S_g\big)$ and, since the functional distortion distance on Reeb graphs corresponds to the first part of the functional distortion distance of barcode transforms (\cref{def:FD_distance_BT}), we obviously get $d_\mathrm{FD}\big(R_f,S_g\big)\leq d_\mathrm{FD}\big(BF,BG\big)$.

    To demonstrate the inequality on the right, let $\mathbf{dom}\colon \mathbf{PVec}\rightarrow \mathbf{Set}$ be the functor that sends functors in $\mathbf{PVec}$ to its domain. We observe that $\mathbf{dom} \hspace{1pt}\mathcal{F}=\mathcal{R}$ the categorical Reeb graph of $(R,f)$. Hence, given an $\epsilon$-interleaving between $\mathcal{F}$ and $\mathcal{G}$, applying $\mathbf{dom}$ yields an $\epsilon$-interleaving between $\mathcal{R}$ and $\mathcal{S}$ and, furthermore, an $\epsilon$-interleaving between $(R,f)$ and $(S,g)$. As shown in \cite{bauer}, $d_\mathrm{FD}(R_f,S_g)\leq 3d_\mathrm{I}(R,S)$. One can now check that the functions $\Phi\colon R\rightarrow S$ and $\Psi\colon S\rightarrow R$ constructed in the proof of this inequality satisfy $ d_\mathrm{I}\big(BF(r),BG\circ \Phi(r)\big)\leq 6\epsilon$ and $\hspace{2pt} d_\mathrm{I}\big(BF\circ\Psi(s),BG(s)\big)\leq 6\epsilon$ for all $r\in R$ and for all $s\in S$.  
\end{proof}

The details of the last part of the proof are quite technical, and we provide full details in the Appendix. 

\begin{remark}\label{rmk:homotopy_type}
    We remark that this result is interesting from a theoretical perspective, but has some shortcomings in practice. In particular, the functional distortion distance used here is infinite if the ranks of $BF(r)$ and $BG(r)$ do not agree for all sufficiently large $r$. 
\end{remark}

\section{Computation}\label{sec:computation}

We now describe constructions of attributed graphs from point cloud data. In the following, we generically refer to such attributed graphs as Decorated Reeb Graphs (DRGs).

\paragraph{Creating Reeb Graphs.} Reeb graphs are most naturally defined for continuous metric spaces, so one needs to approximate a Reeb graph structure for discrete data. We provide a construction similar to the Mapper algorithm \cite{singh2007topological} for estimating Reeb graphs. We first fix a scale $r$ for the Vietoris-Rips complex $VR(X,r)$. Choosing an appropriate value of $r$ is treated as a hyperparameter tuning process; similar ideas for Reeb graph estimation go back at least to~\cite{ge2011data}. For the shape datasets considered in this paper, we used a simple heuristic which took $r = m \cdot r_0$, where $r_0$ is the smallest scale at which the VR complex is connected and $m$ is a small integer (we typically took $m=2$ or $3$). Next, we choose a resolution parameter $n$ and uniformly subdivided the image of $f$ into $n$ bins, $U_1,\ldots,U_n$ (we treat this as a partition of the range, but one could instead thicken slightly and work with an open cover, similar to the usual Mapper construction). This is used to approximate the Reeb graph $G$ of $(VR(X,r),f)$: each node $v$ of $G$ corresponds to a connected component $A_v \subset X$ of one of the sets $f^{-1}(U_i)$, and there is an edge between nodes $v$ and $w$ if $A_v \subset f^{-1}(U_i)$, $A_w \subset f^{-1}(U_{i+1})$ and $A_v$ and $A_w$ are connected by an edge in the 1-skeleton of $VR(X,r)$. An alternative approach would be to compute the exact Reeb graph for $VR(X,r)$ via algorithms of~\cite{harvey2010randomized} or~\cite{parsa2012deterministic}; these algorithms are very efficient, but we found that they did not scale well in the Vietoris-Rips setting due to blowup in the size of the simplicial complexes. 

\paragraph{Creating DRGs - Local Approach.} 

Let $G = (V,E)$ be an estimated Reeb graph for $VR(X,r)$. A simple approach to adding persistent homology decorations to $G$ is as follows. In this \emph{local approach},  degree-$n$ persistent homology is computed for the subset of $X$ corresponding to each node in $G$. The resulting data structure $(G,D)$ consisting of a finite graph $G = (V,E)$ and an attribution function $D:V \to \R \times \cat{Barcodes}$, where $\cat{Barcodes}$ is the set of (persistent homology) barcodes. The attribution takes a node $v \in V$ to $D(v) = (\bar{f}(v),B(v))$, where $\bar{f}(v) = \frac{1}{|A_v|} \sum_{x \in A_v} f(x)$ and $B(v)$ is the persistent homology barcode of $A_v$. 

This method for constructing a DRG can be seen as an approximation of a particular slice of the persistent discrete DRG structure $DF:(\R^2, \leq) \to \cat{Fun}(\mathcal{O}(\R), \cat{PVec})$ introduced in Definition  \ref{def:persistence_discrete_DRG}, as we observed in Remark \ref{rmk:DRGandDMG}.

\paragraph{Creating DRGs - Barcode Transform Approach.} The following is an alternative approach to adding persistent homology decorations to $G = (V,E)$, an estimated Reeb graph for $VR(X,r)$. For each $v \in V$, we define a filtration on $VR(X,r)$ by distance to the set $A_v$ and compute the degree-$n$ persistent homology of the resulting filtered simplicial complex. This once again results in a data structure of the form $(G,D)$ with the attribution function $D:V \to \R \times \cat{Barcodes}$ now recording average function value and persisent homology of the distance-to-$A_v$ function. This method for constructing the DRG is a simplification of the true barcode transform for the decorated Reeb graph of $VR(X,r)$ (see Definition \ref{def:barcode-transform}).

\begin{remark}
The local DRG algorithm easily scales to handle datasets with thousands of points and provides intuitive data summaries consisting of an approximation of the Reeb graph skeleton, attributed with persistence diagrams encoding local structural information (see Figures \ref{fig:intro_example} and  \ref{fig:SeveralExamples}). However, we note that this representation is an aggressive simplification of the structure described in Definition \ref{def:persistence_discrete_DRG}, so that the theoretical stability result of Theorem \ref{thm:persistent_DRGs} does not directly apply in this setting. On the other hand, the barcode transform approach to constructing DRGs gives an approximation of the true barcode transform of $VR(X,r)$, and is therefore much more closely tied to theory. This construction requires several computations of persistent homology on the full complex $VR(X,r)$ (endowed with different filtrations), so that it is not scalable to large datasets. As such, most of our computational experiments below will focus on the local DRG construction.
\end{remark}

\paragraph{Comparing DRGs.} Since the interleaving distance considered in Theorem \ref{thm:persistent_DRGs} is not applicable to our data representation and is computationally intractable, we use the \emph{Fused Gromov-Wasserstein (FGW) framework} \cite{vayer2020fused} for metric-based analysis of DRGs. Intuitively,  the FGW distance, defined below, is  a more easily approximable proxy for the functional distortion distance $d_\mathrm{FD}$.

Let $(G_1,D_1)$, $(G_2,D_2)$ be DRGs constructed as described above (via either the local or barcode transform approaches). The \emph{$\alpha$-FGW distance} is defined by
\begin{align}
&d_{\mathrm{FGW},\alpha}((G_1,D_1),(G_2,D_2))^2 \label{eqn:dfgw} \\
& \qquad \qquad =\inf_{\pi \in \mathcal{C}(V_1,V_2)} \alpha \cdot L_\mathrm{gr}(\pi) + (1-\alpha)\cdot L_\mathrm{bc}(\pi) \nonumber
\end{align}
where the set $\mathcal{C}(G_1,G_2)$ and the loss functions $L_\mathrm{gr}$ and $L_\mathrm{bc}$ are defined as follows. An element $\pi \in \mathcal{C}(G_1,G_2)$ is a matrix of size $|V_1| \times |V_2|$ satisfying $\sum_{v_1 \in V_1} \pi(v_1,v_2) = \frac{1}{|V_2|}$ for all $v_2 \in V_2$ and $\sum_{v_2 \in V_2} \pi(v_1,v_2) = \frac{1}{|V_1|}$ for all $v_1 \in V_1$---intuitively, this is the space of probabilistic couplings of the uniform measures on $V_1$ and $V_2$, respectively. The \emph{graph loss} $L_\mathrm{gr}(\pi)$ is defined by 
\begin{equation}\label{eqn:graph_loss}
\sum \left(d_1(v_1,w_1) - d_2(v_2,w_2)\right)^2 \pi(v_1,v_2)\pi(w_1,w_2),
\end{equation}
 where the sum is over all $v_i,w_i \in V_i$ and  $d_\mathrm{I}:V_i \times V_i \to \R$ is a choice of function representing graph structure---for example, we typically use the shortest path distance, where each edge $(v_i,w_i)$ in $E_i$ is weighted by $|\overline{f}_i(v_i) - \overline{f}_i(w_i)|$. Finally, the \emph{barcode loss} $L_\mathrm{bc}(\pi)$ is defined by 
 \begin{equation}\label{eqn:barcode_loss}
 \sum_{v_i \in V_i} d_b(D_1(v_1),D_2(v_2))^2 \pi(v_1,v_2),
 \end{equation}
 where $d_b$ is the standard bottleneck distance between barcodes. The intuition for the distance is as follows: $\pi \in \mathcal{C}(V_1,V_2)$ is interpreted as a probabilistic registration of the nodes of $G_1$ and $G_2$, the loss $L_{\mathrm{gr}}$ measures how well the registration preserves the graph structure, the loss $L_{\mathrm{bc}}$ measures how well the registration preserves attributions, and the hyperparameter $\alpha$ balances contributions of graph structure and attributions; the optimization problem therefore searches for a probabilistic registration which incurs the least total distortion of these structures. Fixing a methodology for assigning a distance graph function $d:V \times V \to \R$ to a DRG $(G,D)$, it follows from Theorem 1 of \cite{vayer2020fused} that $d_{\mathrm{FGW},\alpha}$ defines a pseudometric on the space of DRGs, for any choice of $\alpha \in [0,1]$. 

The idea for FGW distance originates from the Gromov-Wasserstein (GW) distances introduced by M\'{e}moli in \cite{10.2312:SPBG:SPBG07:081-090}; roughly, the GW 2-distance is obtained by setting $\alpha=1$ in the FGW formula. The GW distances can be seen as $L^p$-relaxations of the Gromov-Hausdorff (GH) distance and are used for the comparison of metric measure spaces (mm-spaces). The FGW distance was introduced in \cite{vayer2020fused} to adapt GW distance to the setting of mm-spaces whose points come with feature attributions. In particular, our use of $d_{\mathrm{FGW},\alpha}$ can be seen as a relaxation of the functional distortion distance of Definition \ref{def:FD_distance_BT}. Such a relaxation makes the discrete optimization problem arising in GH distances amenable to approximation by gradient descent, as the search space of the optimization becomes a compact, convex polytope $\mathcal{C}(V_1,V_2) \subset \R^{|V_1| \times |V_2|}$. For finite mm-spaces of size $O(n)$, a gradient descent iteration for approximation of GW distance has $O(n^3\log(n))$ cost~\cite{peyre2016gromov}. For the FGW distance in our setting, computation is complicated by the need to evaluate $|V_1| \times |V_2|$ bottleneck distances, each of which carries a cubic cost in the number of points in the diagrams. To ease this computational burden, we frequently replace the bottleneck distance computations with Euclidean distance between persistence image vectorizations of the diagrams~\cite{adams2017persistence}. We remark that (Fused) Gromov-Wasserstein distances have been used successfully in several recent works to compare other topological invariants---see, e.g.,~\cite{li2021sketching,curry2022decorated,li2023flexible}.

\paragraph{Related Constructions.} Our constructions have a similar flavor to other enriched topological invariants in the literature. Most similar are the Decorated Merge Trees (DMTs) introduced by the first four authors in~\cite{curry2022decorated}. A DMT is a certain rooted tree attributed with homological information which can be extracted from a dataset endowed with a filter function---roughly, a DMT captures the topology of sublevel sets of the filtration, while the constructions in the present paper are concerned with level set topology. Our constructions also share features of Persistent Homology Transforms (PHTs)~\cite{turner2014persistent}, which associate a collection of persistence diagrams to a shape in Euclidean space by using projections to various 1-dimensional subspaces as filter functions. Our constructions also yield families of persistence diagrams, but the families are parameterized by nodes of Mapper graphs, rather than by collections of lines. 

\section{Experiments}\label{sec:examples}

We now illustrate our computational pipeline with several experiments. Our source code is available at our GitHub repository\footnote{https://github.com/trneedham/Topologically-Attributed-Graphs}.

\paragraph{Example DRGs.} In Figure \ref{fig:SeveralExamples}, we provide a few examples of DRGs constructed via the methods described in Section \ref{sec:computation}. 

The first row of Figure \ref{fig:SeveralExamples} shows a shape from ModelNet10~\cite{wu20153d}, a curated collection of CAD models of household objects. The CAD model has been converted to point cloud data by sampling. We show the DRG computed via the local approach, with respect to height along the $z$-axis; nodes of the DRG are colored by total persistence of their diagram attributes. We show the diagrams associated to two of the nodes, as well as the associated subsets of the original dataset. 

The second row of the figure shows a humanoid figure from the SHREC14 dataset~\cite{Pickup2014}; once again, the pointcloud data is obtained by sampling a triangulated surface. For this example, the filtration function is the $p$-eccentricity (with $p=100$) (see~\cite{memoli2011gromov}, Definition 5.3) of the shortest path distance on the 1-skeleton of the underlying Vietoris-Rips complex. We show the associated DRG (via the local approach), some persistence diagram attributes, and the persistence image vectorizations of these diagrams. 

Finally, the third row of Figure \ref{fig:SeveralExamples} shows the synthetic dataset from Figure \ref{fig:intro_example}, once again endowed with the height function. In this case, the DRG is computed via the Barcode Transform approach. Observe that the associated persistence diagrams capture the global topology of the shape---note that the death time of each point is $\infty$ and that one point in each diagram is of multiplicity 2. The difference between the diagrams is the birth times of the features; the DRG nodes are colored by average birth time in their diagrams.

\begin{figure}
    \centering
    \includegraphics[width = 0.48\textwidth]{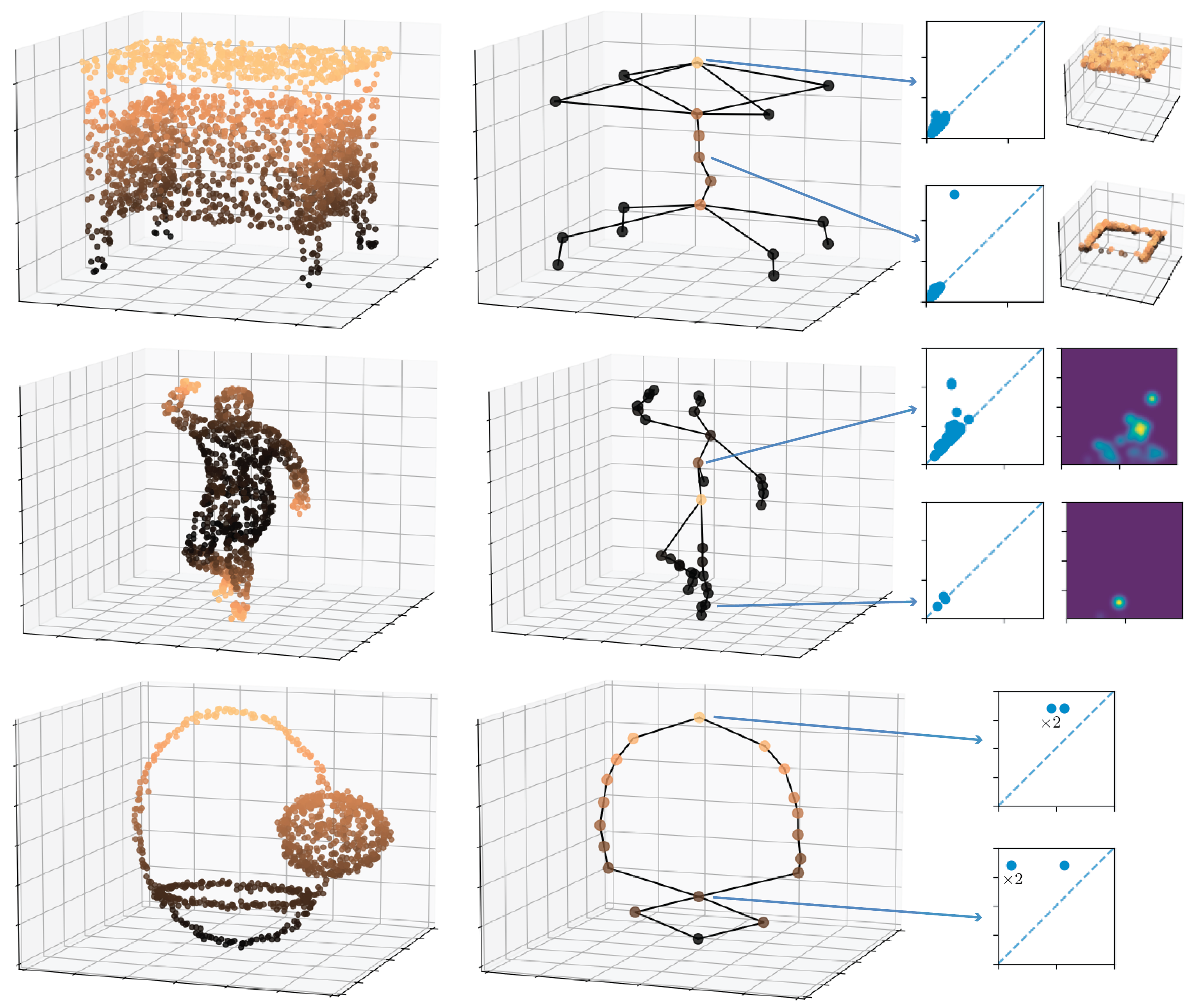}
    \caption{Examples of DRGs created with our methods. See Section \ref{sec:examples} for a detailed description.}
    \label{fig:SeveralExamples}
\end{figure}

\paragraph{Synthetic Shape Comparison.} To explore the behavior of the Fused Gromov-Wasserstein (FGW) distance $d_{\mathrm{FGW},\alpha}$ defined in \eqref{eqn:dfgw}, we consider a synthetic point cloud  dataset consisting of four classes: torus, solid torus, cylinder and solid cylinder. Each torus shape consists of 400 points sampled from a toroidal surface with minor radius 1 and major radius 6 and solid tori are generated similarly, but we take 1600 samples to get a comparable density. Cylinders consist of 400 points sampled from the surface of a cylinder of radius 1 and length $2\cdot \pi \cdot (6+1)$ so that the surface area is comparable to the torus and solid cylinders are generated similarly with 1600 sample points. Each shape in the dataset has its point coordinates perturbed independently at random and the resulting point cloud is then randomly rigidly rotated. The full dataset consists of 20 samples of each shape class. 

Each shape is converted to a DRG using the local approach, with filter functions given by 1st PCA coordinates. Node attributes are converted to persistence images, for computational efficiency. For each $\alpha \in \{0.0,0.25,0.5,1.0\}$, we construct the shape-to-shape distance matrix with respect to $d_{\mathrm{FGW},\alpha}$, where the $d_\mathrm{I}$'s in \eqref{eqn:graph_loss} are shortest path distances in the Reeb graph and $d_b$ in \eqref{eqn:barcode_loss} is replaced by Euclidean distance between persistence images. 

Multidimensional Scaling plots for the distance matrices are shown in \ref{fig:ClusteringMDS}. When $\alpha = 0$, the distance only sees the persistence image attributes, and the torus and cylinder classes are indistinguishable (the topology of each connected component of each of their level sets are very similar). For $\alpha \in \{0.25,0.5\}$, both the global structure of the Reeb graph and the local persistence image structures are considered, and all classes are able to be distinguished. Finally, when $\alpha = 1$, only the Reeb graph structures are considered in the metric; here, the torus and solid torus classes are unable to be distinguished, and the same is true of the cylinder and solid cylinder classes. This experiment suggests that there is a fairly robust range of $\alpha$-values where the metric meaningfully takes into account both Reeb graph and local persistence structures for distinguishing shapes. 

We remark that we ran the same experiment using DRGs constructed via the Barcode Transform method; here the results were much less interpretable. This is due to the fact that the distance between Barcode Transform DRGs is strongly controlled by the homotopy types of the shapes (cf. Remark \ref{rmk:homotopy_type}); it was difficult to tune parameters so that the underlying Vietoris-Rips complexes used in the construction consistently had the correct homotopy type. 

\begin{figure}
    \centering
    \includegraphics[width = 0.48\textwidth]{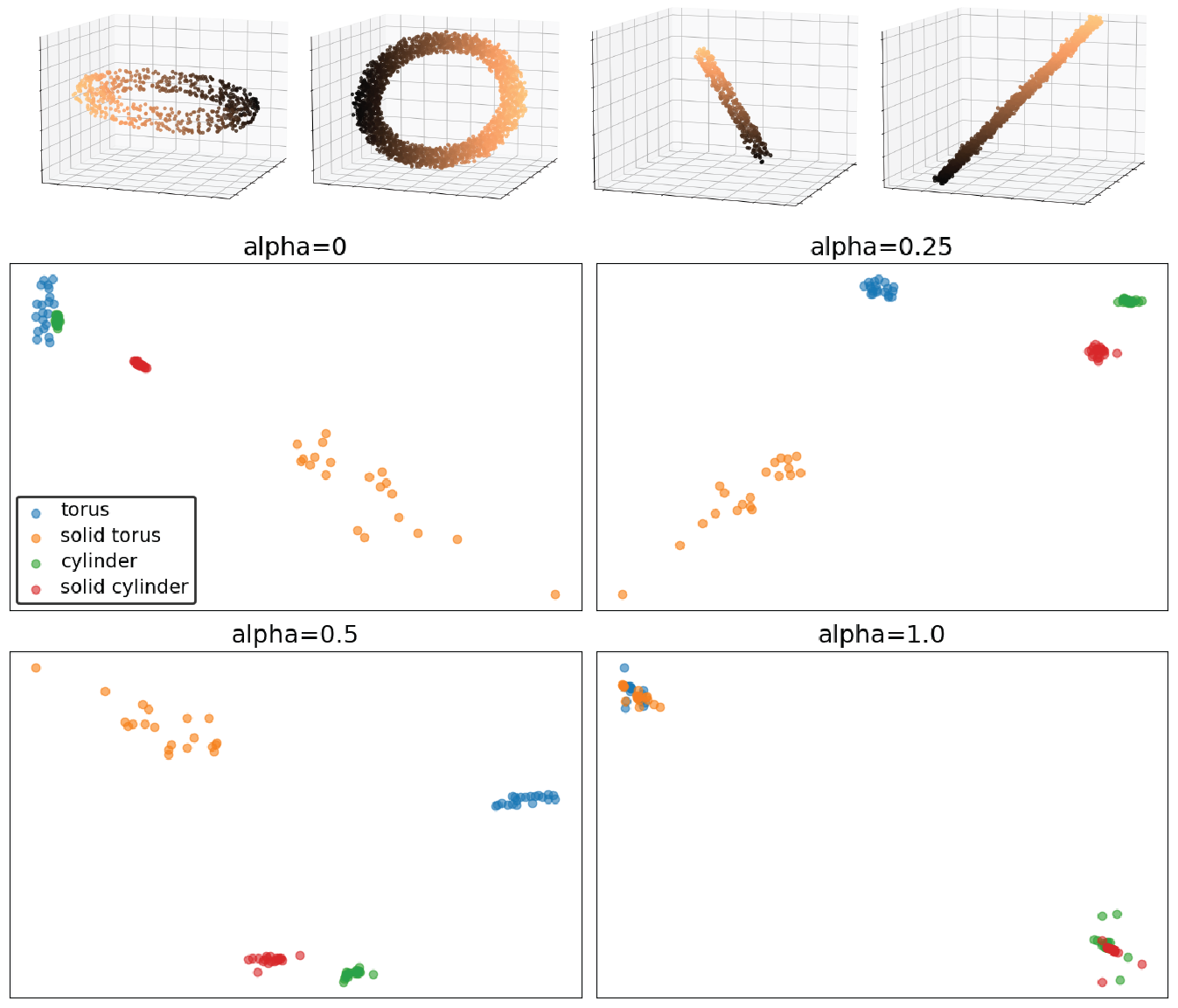}
    \caption{Synthetic shape dataset results. The top row shows a sample from each of the shape classes in the experiment; points are colored by their filter function values (first PCA coordinate). The remaining figures show Multidimensional Scaling embeddings of the dataset, coming from pairwise distance matrices with respect to $d_{\mathrm{FGW},\alpha}$ for various $\alpha$-values.}
    \label{fig:ClusteringMDS}
\end{figure}

\paragraph{Graph Neural Network-Based Classification.}

To more thoroughly test the capability of DRGs to distinguish shapes, we use DRGs as inputs to a Graph Neural Network (GNN) classifier. Our data comes from the ModelNet10 dataset~\cite{wu20153d} of CAD models of household objects. The data consists of 10 classes of objects, pre-partitioned into a training set (3991 objects) and a testing set (908 objects). We sample 1024 points form each object to form a dataset of point clouds. From each point cloud, we extract two DRGs: one using $z$-coordinates as the filtration function and the other using $x$-coordinates (see the first row of Figure \ref{fig:SeveralExamples}). The nodes of the DRGs are attributed with vectorizations of the persistence diagrams---in this case, we used summary statistics of the persistence diagrams, following \cite{ali2022survey}, Definition 2.1. We also attributed each node with the average Euclidean position of its associated pointcloud. These vector-attributed graphs were used as input to a simple GNN, consisting of four convolutional layers of width 256, implemented in PyTorch. All neural networks in this experiment were trained on a single CPU.

\begin{table}[htbp]
  \centering
  \caption{ModelNet10 Classification Results.}
  \label{tab:modelnet10Results}
  \begin{tabular}{c|ccccc}
    \toprule
    \small{PointNet} & \small{Reeb} & \small{Dgms} & \small{$\mathrm{DRG}_z$} & \small{$\mathrm{DRG}_x$} & \small{$\mathrm{DRG}_{\mathrm{xz}}$} \\
    \midrule
    89.43 & 77.64 & 63.55 & 84.69 & 85.24 & 87.11 \\
    \bottomrule
  \end{tabular}
\end{table}

Results of the classification experiment are reported in Table \ref{tab:modelnet10Results}. Besides results for the DRGs with filter function given by $z$-coordinate ($\mathrm{DRG}_z$) and $x$-coordinate ($\mathrm{DRG}_x$), we also report the combined prediction from the two models ($\mathrm{DRG}_{xz}$). The combined prediction was made by averaging the predictions of the $\mathrm{DRG}_z$ and $\mathrm{DRG}_x$ models---we also trained a GNN using a disjoint pair of DRGs (with $x$ and $z$-coordinate filtrations) to represent each shape and got similar classification scores. To test the contributions of the diagram attributes, we also trained a network on graphs where nodes were only attributed with average Euclidean coordinates ($\mathrm{Reeb}$). Likewise, we tested the contribution of the graph structure by converting each DRG into a complete graph on its nodes, each of which is attributed with persistence statistics and Euclidean coordinates ($\mathrm{Dgms}$). We see that the combination of graph structure and topological attributes provide a large boost in classification accuracy, with the best accuracy obtained by the combination of $x$ and $z$-filtrations.

To test against a baseline, we use the popular PointNet architecture for point cloud classification~\cite{qi2017pointnet}. The PointNet classification accuracy is essentially state-of-the-art for ModelNet10 classification (when using only point clouds as input, without additional structure from the CAD models), and we see that it has a slight edge over the DRG classification score. However, we note that the PointNet model\footnote{A PyTorch implementation, following \url{https://github.com/nikitakaraevv/pointnet}.} contains 3,463,763 parameters, compared to our 209,162 parameter GNN. Moreover, achieving this level of accuracy took $\sim$12 hours of training time for PointNet, while our GNN model contained an order of  trained in $\sim$10 minutes; we preprocessed the data to extract Reeb graphs, which took $\sim$1.5 hours, bringing the total time for processing and training to around $\sim$3 hours (processing using both $x$ and $z$ filtrations). This suggests that the DRG representations have a rich structure with easily learnable features.

\begin{table}[htbp]
  \centering
  \caption{ModelNet10 Subset Classification Results.}
  \label{tab:modelnet10SubsetResults}
  \begin{tabular}{c|ccc}
    \toprule
    \small{PointNet} & \small{$\mathrm{DRG}_z$} & \small{$\mathrm{DRG}_x$} & \small{$\mathrm{DRG}_{\mathrm{xz}}$} \\
    \midrule
    83.26 & 78.52 & 70.70 & 82.93 \\
    \bottomrule
  \end{tabular}
\end{table}

We tested the representational richness of DRGs further by retraining the DRG and PointNet models on only 10\% of the ModelNet10 training data, then testing classification accuracy on the full training set. In this sparse training data setting, we see that all models still perform reasonably well, but point out that the gap between PointNet and $\mathrm{DRG}_{xz}$ has essentially vanished, even though the latter is less complex by an order of magnitude.

\section{Discussion} 

In this paper, we introduced formalism for topologically attributed graphs and provided theoretical results on their stability. We also demonstrated the potential applicability of these ideas through proof-of-concept experiments. Future work will involve building a closer connection between the theory of topologically attributed Reeb graphs and their computational execution. Notably, our computational pipeline does not incorporate the more sheaf-theoretic or categorical features of decorated Reeb graphs, and integration of these aspects is an important goal. We also plan to continue to develop the computational pipeline toward more robust applications. One interesting direction will be to develop the pipeline to handle more general filtration functions, or to incorporate discovery of effective filter functions into a machine learning framework. We also intend to extend this framework to handle more general attributed simplicial complexes, to which newly developed tools of Topological Deep Learning~\cite{hajij2023topological} will apply.

\bibliography{icml_submission}

\begin{thebibliography}{30}
\providecommand{\natexlab}[1]{#1}
\providecommand{\url}[1]{\texttt{#1}}
\expandafter\ifx\csname urlstyle\endcsname\relax
  \providecommand{\doi}[1]{doi: #1}\else
  \providecommand{\doi}{doi: \begingroup \urlstyle{rm}\Url}\fi

\bibitem[Adams et~al.(2017)Adams, Emerson, Kirby, Neville, Peterson, Shipman,
  Chepushtanova, Hanson, Motta, and Ziegelmeier]{adams2017persistence}
Adams, H., Emerson, T., Kirby, M., Neville, R., Peterson, C., Shipman, P.,
  Chepushtanova, S., Hanson, E., Motta, F., and Ziegelmeier, L.
\newblock Persistence images: A stable vector representation of persistent
  homology.
\newblock \emph{Journal of Machine Learning Research}, 18, 2017.

\bibitem[Ali et~al.(2022)Ali, Asaad, Jimenez, Nanda, Paluzo-Hidalgo, and
  Soriano-Trigueros]{ali2022survey}
Ali, D., Asaad, A., Jimenez, M.-J., Nanda, V., Paluzo-Hidalgo, E., and
  Soriano-Trigueros, M.
\newblock A survey of vectorization methods in topological data analysis.
\newblock \emph{arXiv preprint arXiv:2212.09703}, 2022.

\bibitem[Bauer et~al.(2014)Bauer, Ge, and Wang]{bauer2014measuring}
Bauer, U., Ge, X., and Wang, Y.
\newblock Measuring distance between {R}eeb graphs.
\newblock In \emph{Proceedings of the thirtieth annual symposium on
  Computational geometry}, pp.\  464--473, 2014.

\bibitem[Bauer et~al.(2015)Bauer, Munch, and Wang]{bauer}
Bauer, U., Munch, E., and Wang, Y.
\newblock {Strong Equivalence of the Interleaving and Functional Distortion
  Metrics for {R}eeb Graphs}.
\newblock In Arge, L. and Pach, J. (eds.), \emph{31st International Symposium
  on Computational Geometry (SoCG 2015)}, volume~34 of \emph{Leibniz
  International Proceedings in Informatics (LIPIcs)}, pp.\  461--475, Dagstuhl,
  Germany, 2015. Schloss Dagstuhl--Leibniz-Zentrum fuer Informatik.
\newblock ISBN 978-3-939897-83-5.
\newblock \doi{10.4230/LIPIcs.SOCG.2015.461}.
\newblock URL \url{http://drops.dagstuhl.de/opus/volltexte/2015/5146}.

\bibitem[Carriere \& Oudot(2018)Carriere and Oudot]{carriere2018structure}
Carriere, M. and Oudot, S.
\newblock Structure and stability of the one-dimensional {M}apper.
\newblock \emph{Foundations of Computational Mathematics}, 18:\penalty0
  1333--1396, 2018.

\bibitem[Chazal et~al.(2009)Chazal, Cohen-Steiner, Guibas, M{\'e}moli, and
  Oudot]{chazal2009gromov}
Chazal, F., Cohen-Steiner, D., Guibas, L.~J., M{\'e}moli, F., and Oudot, S.~Y.
\newblock Gromov-{H}ausdorff stable signatures for shapes using persistence.
\newblock In \emph{Computer Graphics Forum}, volume~28, pp.\  1393--1403. Wiley
  Online Library, 2009.

\bibitem[Chazal et~al.(2014)Chazal, De~Silva, and Oudot]{chazal2014persistence}
Chazal, F., De~Silva, V., and Oudot, S.
\newblock Persistence stability for geometric complexes.
\newblock \emph{Geometriae Dedicata}, 173\penalty0 (1):\penalty0 193--214,
  2014.

\bibitem[Curry et~al.(2022)Curry, Hang, Mio, Needham, and
  Okutan]{curry2022decorated}
Curry, J., Hang, H., Mio, W., Needham, T., and Okutan, O.~B.
\newblock Decorated merge trees for persistent topology.
\newblock \emph{Journal of Applied and Computational Topology}, 6\penalty0
  (3):\penalty0 371--428, 2022.

\bibitem[de~Silva et~al.(2016)de~Silva, Munch, and Patel]{desilva}
de~Silva, V., Munch, E., and Patel, A.
\newblock Categorified {R}eeb graphs.
\newblock \emph{Discrete \& Computational Geometry}, 55\penalty0 (4):\penalty0
  854--906, 2016.
\newblock \doi{10.1007/s00454-016-9763-9}.
\newblock URL \url{https://doi.org/10.1007/s00454-016-9763-9}.

\bibitem[De~Silva et~al.(2018)De~Silva, Munch, and Stefanou]{de2018theory}
De~Silva, V., Munch, E., and Stefanou, A.
\newblock Theory of interleavings on categories with a flow.
\newblock \emph{Theory and Applications of Categories}, 33\penalty0
  (21):\penalty0 583--607, 2018.

\bibitem[Ge et~al.(2011)Ge, Safa, Belkin, and Wang]{ge2011data}
Ge, X., Safa, I., Belkin, M., and Wang, Y.
\newblock Data skeletonization via {R}eeb graphs.
\newblock \emph{Advances in Neural Information Processing Systems}, 24, 2011.

\bibitem[Hajij et~al.(2023)Hajij, Zamzmi, Papamarkou, Miolane, Guzmán-Sáenz,
  Ramamurthy, Birdal, Dey, Mukherjee, Samaga, Livesay, Walters, Rosen, and
  Schaub]{hajij2023topological}
Hajij, M., Zamzmi, G., Papamarkou, T., Miolane, N., Guzmán-Sáenz, A.,
  Ramamurthy, K.~N., Birdal, T., Dey, T.~K., Mukherjee, S., Samaga, S.~N.,
  Livesay, N., Walters, R., Rosen, P., and Schaub, M.~T.
\newblock Topological deep learning: Going beyond graph data.
\newblock \emph{arXiv preprint arXiv:2206.00606}, 2023.

\bibitem[Harvey et~al.(2010)Harvey, Wang, and Wenger]{harvey2010randomized}
Harvey, W., Wang, Y., and Wenger, R.
\newblock A randomized {O}(m log m) time algorithm for computing {R}eeb graphs
  of arbitrary simplicial complexes.
\newblock In \emph{Proceedings of the twenty-sixth annual symposium on
  Computational geometry}, pp.\  267--276, 2010.

\bibitem[Hatcher(2002)]{hatcher2002algebraic}
Hatcher, A.
\newblock \emph{Algebraic Topology}.
\newblock Cambridge University Press, 2002.

\bibitem[Li et~al.(2021)Li, Palande, Yan, and Wang]{li2021sketching}
Li, M., Palande, S., Yan, L., and Wang, B.
\newblock Sketching merge trees for scientific data visualization.
\newblock \emph{arXiv preprint arXiv:2101.03196}, 2021.

\bibitem[Li et~al.(2023)Li, Yan, Yan, Needham, and Wang]{li2023flexible}
Li, M., Yan, X., Yan, L., Needham, T., and Wang, B.
\newblock Flexible and probabilistic topology tracking with partial optimal
  transport.
\newblock \emph{arXiv preprint arXiv:2302.02895}, 2023.

\bibitem[M\'{e}moli(2007)]{10.2312:SPBG:SPBG07:081-090}
M\'{e}moli, F.
\newblock {On the use of {G}romov-{H}ausdorff Distances for Shape Comparison}.
\newblock In Botsch, M., Pajarola, R., Chen, B., and Zwicker, M. (eds.),
  \emph{Eurographics Symposium on Point-Based Graphics}. The Eurographics
  Association, 2007.
\newblock ISBN 978-3-905673-51-7.
\newblock \doi{10.2312/SPBG/SPBG07/081-090}.

\bibitem[M{\'e}moli(2011)]{memoli2011gromov}
M{\'e}moli, F.
\newblock Gromov--{W}asserstein distances and the metric approach to object
  matching.
\newblock \emph{Foundations of computational mathematics}, 11:\penalty0
  417--487, 2011.

\bibitem[Munch \& Wang(2016)Munch and Wang]{munch2016convergence}
Munch, E. and Wang, B.
\newblock Convergence between categorical representations of {R}eeb space and
  {M}apper.
\newblock In \emph{32nd International Symposium on Computational Geometry (SoCG
  2016)}. Schloss Dagstuhl-Leibniz-Zentrum fuer Informatik, 2016.

\bibitem[Parsa(2012)]{parsa2012deterministic}
Parsa, S.
\newblock A deterministic {O}(m log m) time algorithm for the {R}eeb graph.
\newblock In \emph{Proceedings of the twenty-eighth annual symposium on
  Computational geometry}, pp.\  269--276, 2012.

\bibitem[Peyr{\'e} et~al.(2016)Peyr{\'e}, Cuturi, and Solomon]{peyre2016gromov}
Peyr{\'e}, G., Cuturi, M., and Solomon, J.
\newblock Gromov-{W}asserstein averaging of kernel and distance matrices.
\newblock In \emph{International conference on machine learning}, pp.\
  2664--2672. PMLR, 2016.

\bibitem[Pickup et~al.(2014)Pickup, Sun, Rosin, Martin, Cheng, Lian, Aono,
  Ben~Hamza, Bronstein, Bronstein, Bu, Castellani, Cheng, Garro, Giachetti,
  Godil, Han, Johan, Lai, Li, Li, Li, Litman, Liu, Liu, Lu, Tatsuma, and
  Ye]{Pickup2014}
Pickup, D., Sun, X., Rosin, P.~L., Martin, R.~R., Cheng, Z., Lian, Z., Aono,
  M., Ben~Hamza, A., Bronstein, A., Bronstein, M., Bu, S., Castellani, U.,
  Cheng, S., Garro, V., Giachetti, A., Godil, A., Han, J., Johan, H., Lai, L.,
  Li, B., Li, C., Li, H., Litman, R., Liu, X., Liu, Z., Lu, Y., Tatsuma, A.,
  and Ye, J.
\newblock S{H}{R}{E}{C}'14 track: Shape retrieval of non-rigid 3d human models.
\newblock In \emph{Proceedings of the 7th Eurographics workshop on 3D Object
  Retrieval}, EG 3DOR'14. Eurographics Association, 2014.

\bibitem[Qi et~al.(2017)Qi, Su, Mo, and Guibas]{qi2017pointnet}
Qi, C.~R., Su, H., Mo, K., and Guibas, L.~J.
\newblock Pointnet: Deep learning on point sets for 3d classification and
  segmentation.
\newblock In \emph{Proceedings of the IEEE conference on computer vision and
  pattern recognition}, pp.\  652--660, 2017.

\bibitem[Reeb(1946)]{reeb1946points}
Reeb, G.
\newblock Sur les points singuliers d'une forme de pfaff completement
  integrable ou d'une fonction numerique [on the singular points of a
  completely integrable pfaff form or of a numerical function].
\newblock \emph{Comptes Rendus Acad. Sciences Paris}, 222:\penalty0 847--849,
  1946.

\bibitem[Riehl(2017)]{riehl2017category}
Riehl, E.
\newblock \emph{Category theory in context}.
\newblock Courier Dover Publications, 2017.

\bibitem[Singh et~al.(2007)Singh, M{\'e}moli, Carlsson,
  et~al.]{singh2007topological}
Singh, G., M{\'e}moli, F., Carlsson, G.~E., et~al.
\newblock Topological methods for the analysis of high dimensional data sets
  and 3d object recognition.
\newblock \emph{PBG@ Eurographics}, 2:\penalty0 091--100, 2007.

\bibitem[Stefanou(2018)]{stefanou2018dynamics}
Stefanou, A.
\newblock \emph{Dynamics on categories and applications}.
\newblock State University of New York at Albany, 2018.

\bibitem[Turner et~al.(2014)Turner, Mukherjee, and Boyer]{turner2014persistent}
Turner, K., Mukherjee, S., and Boyer, D.~M.
\newblock Persistent homology transform for modeling shapes and surfaces.
\newblock \emph{Information and Inference: A Journal of the IMA}, 3\penalty0
  (4):\penalty0 310--344, 2014.

\bibitem[Vayer et~al.(2020)Vayer, Chapel, Flamary, Tavenard, and
  Courty]{vayer2020fused}
Vayer, T., Chapel, L., Flamary, R., Tavenard, R., and Courty, N.
\newblock Fused {G}romov-{W}asserstein distance for structured objects.
\newblock \emph{Algorithms}, 13\penalty0 (9):\penalty0 212, 2020.

\bibitem[Wu et~al.(2015)Wu, Song, Khosla, Yu, Zhang, Tang, and Xiao]{wu20153d}
Wu, Z., Song, S., Khosla, A., Yu, F., Zhang, L., Tang, X., and Xiao, J.
\newblock 3d {S}hapenets: A deep representation for volumetric shapes.
\newblock In \emph{Proceedings of the IEEE conference on computer vision and
  pattern recognition}, pp.\  1912--1920, 2015.

\end{thebibliography}
\bibliographystyle{icml2023}

\clearpage

\newpage

\appendix

\section{Proof of Theorem \ref{thm:strong_equivalence}}

\textbf{Left inequality} \\
As shown in \cite{bauer}, $d_\mathrm{I}\big(R,S\big)\leq d_\mathrm{FD}\big(R_f,S_g\big)$ and, since the functional distortion distance on Reeb graphs corresponds to the first part of the functional distortion distance of barcode transforms (\cref{def:FD_distance_BT}), we obviously get $d_\mathrm{FD}\big(R_f,S_g\big)\leq d_\mathrm{FD}\big(BF,BG\big)$. \\

\textbf{Right inequality} \\
Let $S_\epsilon\colon \cat{Fun}(\mathcal{O}(\R), \cat{PVec})\rightarrow \cat{Fun}(\mathcal{O}(\R), \cat{PVec})$ be the smoothing functor, defined by $S_\epsilon\mathcal{F}(U)\coloneqq \mathcal{F}(U^\epsilon)$ and $\iota\colon\mathcal{F}\rightarrow S_\epsilon \mathcal{F}$ be defined by $\iota(U)\coloneqq\mathcal{F}(U\subseteq U^\epsilon)$. In the following we denote an object of $\mathbf{PVec}$ by a tuple $(I,D)$ representing a set $I$ and a functor $D\colon I\rightarrow \mathbf{Vec}$. Suppose $\mathcal{F}$ and $\mathcal{G}$ are $\epsilon$-interleaved, i.e.\ we have the following commutative diagram:
\begin{equation} \label{cat_interleaving}
\begin{tikzcd}[column sep=huge,row sep=huge]
\mathcal{F} \arrow[r,"\iota"] \arrow[dr,swap,"(\alpha\text{,}\eta)"{xshift=-12pt,yshift=10pt}] & S_\epsilon\mathcal{F} \arrow[r,"\iota_{\epsilon}"] \arrow[dr,swap,"(\alpha_\epsilon\text{,}\eta_\epsilon)"{xshift=-12pt,yshift=10pt}] & S_{2\epsilon}\mathcal{F} \\
\mathcal{G} \arrow[r,"\iota"] \arrow[ur,crossing over,swap,"(\beta\text{,}\rho)"{xshift=12pt,yshift=10pt}] & S_\epsilon\mathcal{G} \arrow[ur,swap,crossing over,"(\beta_\epsilon\text{,}\rho_\epsilon)"{xshift=12pt,yshift=10pt}] \arrow[r,"\iota_{\epsilon}"] & S_{2\epsilon}\mathcal{G}
\end{tikzcd}
\end{equation}
where $\alpha$ denotes the morphisms between the parameterizing sets and $\eta$ denotes the morphisms between the parameterized vector spaces.
Let $\mathbf{dom}\colon \mathbf{PVec}\rightarrow \mathbf{Set}$ be a forgetful functor defined on an object $(I,D)\in\mathbf{PVec}$ by $\mathbf{dom}\big((I,D)\big)\coloneqq I$ and on a morphism $(\alpha,\eta)\colon (I,D)\rightarrow (I',D')$ by $\mathbf{dom}\big((\alpha,\eta)\big)\coloneqq \alpha$. If we postcompose $\mathcal{F}$ with $\mathbf{dom}$ we obtain $\mathbf{dom}\circ\mathcal{F}(U)=\mathbf{dom}(\mathcal{F}(U))=\pi_0(f^{- 1}(U))$ and $\mathbf{dom}\circ\mathcal{F}(U\subseteq V)=\pi_0(f^{-1}(U)\subseteq f^{-1}(V))$. Hence, we get $\mathbf{dom}\circ \mathcal{F}=\mathcal{R}$ the categorical Reeb graph corresponding to $(R,f)$. Denote by  $\mathbf{R}\colon\mathbf{concrete Reeb graphs}\rightarrow\mathbf{categorical Reeb graphs}$ the functor that sends a concrete Reeb graph $(R,f)$ to the corresponding categorical Reeb graph $\mathcal{R}$ (see \cite{desilva}). We now apply $\mathbf{dom}$ on \cref{cat_interleaving} and obtain the following commutative diagram of $\mathbf{Set}$-valued functors:
\begin{equation} \label{set_interleaving_1}
\begin{tikzcd}[column sep=huge,row sep=huge]
\mathbf{R}(R,f) \arrow[r,"\iota"] \arrow[dr,swap,"\alpha"{xshift=-12pt,yshift=10pt}] &[-20pt] S_\epsilon\mathbf{R}(R,f) \arrow[r,"\iota_{\epsilon}"] \arrow[dr,swap,"\alpha_\epsilon"{xshift=-12pt,yshift=10pt}] &[-20pt] S_{2\epsilon}\mathbf{R}(R,f) \\
\mathbf{R}(S,g) \arrow[r,"\iota"] \arrow[ur,crossing over,swap,"\beta"{xshift=12pt,yshift=10pt}] & S_\epsilon\mathbf{R}(S,g) \arrow[ur,swap,crossing over,"\beta_\epsilon"{xshift=12pt,yshift=10pt}] \arrow[r,"\iota_{\epsilon}"] & S_{2\epsilon}\mathbf{R}(S,g)
\end{tikzcd}
\end{equation}
By Proposition 4.29 in \cite{desilva}, the smoothing of open sets $S_\epsilon$ is equivalent to the smoothing of the underlying geometric Reeb graphs. Let $T_\epsilon(R,f)$ be the $\epsilon$-thickening of $(R,f)$ defined by $T_\epsilon R\coloneqq R\times [-\epsilon,\epsilon]$ and $\mathcal{U}_\epsilon(R,f)$ be the Reeb graph of $T_\epsilon(R,f)$ (the $\epsilon$-smoothing of $(R,f)$). These spaces can be summarized by the following commutative diagram:
\begin{equation}
\begin{tikzcd}[column sep=large,row sep=large]
R \arrow[dr,swap,"f"] & \arrow[l,swap,"p_1"] T_\epsilon R \arrow[d,"\hat{f}_\epsilon"] \arrow[r,"q"] & \mathcal{U}_\epsilon R \arrow[dl,"f_\epsilon"] \\
& \mathbb{R}
\end{tikzcd}
\end{equation} 
where $p_1$ is the projection to the first factor and $q$ is the quotient map to the Reeb space. The map $p_1$ induces a natural isomorphism $\mathbf{R}T_\epsilon\implies S_\epsilon\mathbf{R}$ such that
\begin{equation*}
\begin{aligned}
&\Big(\mathbf{R}T_\epsilon(R,f)(U)\xrightarrow{} S_\epsilon \mathbf{R}(R,f)(U)\Big)  \\ =&\Big(\pi_0(\hat{f}_\epsilon^{-1}(U))\xrightarrow{\pi_0(p_1)}\pi_0(f^{-1}(U^\epsilon))\Big)
\end{aligned}
\end{equation*}; \cite{desilva} Theorem
4.2. Moreover, the map $q$ induces a natural isomorphism $\mathbf{R}T_\epsilon\implies \mathbf{R}\mathcal{U}_\epsilon$ such that 
\begin{equation*}
\begin{aligned}
&\Big(\mathbf{R}T_\epsilon(R,f)(U)\rightarrow \mathbf{R}\mathcal{U}_\epsilon(R,f)(U)\Big) \\ =&\Big(\pi_0(\hat{f}_\epsilon^{-1}(U))\xrightarrow{\pi_0(q)}\pi_0(f_\epsilon^{-1}(U))\Big)
\end{aligned}
\end{equation*}; \cite{desilva} Theorem 3.15. Let $h$ denote the composition of the following natural isomorphisms: 
\begin{equation} \label{hiso}
\begin{tikzcd}
h\colon S_\epsilon\mathbf{R} \arrow[r,Rightarrow]  &[-40pt] \mathbf{R}T_\epsilon \arrow[r,Rightarrow] &[-40pt] \mathbf{R}\mathcal{U}_\epsilon \\
h(U)\colon\pi_0(f^{-1}(U^\epsilon)) \arrow[dr,swap,"\pi_0(p_1)^{-1}"] & &  \pi_0(f_\epsilon^{-1}(U)) \\ & \pi_0(\hat{f}_\epsilon^{-1}(U)) \arrow[ur,swap,"\pi_0(q)"] 
\end{tikzcd}
\end{equation}
Applying $h$ to \cref{set_interleaving_1} yields
\begin{equation} \label{set_interleaving_2}
\begin{tikzcd}[column sep=huge,row sep=huge]
\mathbf{R}(R,f) \arrow[r,"h(\iota)"] \arrow[dr,swap,"h(\alpha)"{xshift=-12pt,yshift=10pt}] &[-20pt] \mathbf{R}\mathcal{U}_\epsilon(R,f) \arrow[r,"h(\iota_{\epsilon})"] \arrow[dr,swap,"h(\alpha_\epsilon)"{xshift=-12pt,yshift=10pt}] &[-20pt] \mathbf{R}\mathcal{U}_{2\epsilon}(R,f) \\
\mathbf{R}(S,g) \arrow[r,"h(\iota)"] \arrow[ur,crossing over,swap,"h(\beta)"{xshift=12pt,yshift=10pt}] & \mathbf{R}\mathcal{U}_\epsilon(S,g) \arrow[ur,swap,crossing over,"h(\beta_\epsilon)"{xshift=12pt,yshift=10pt}] \arrow[r,"h(\iota_{\epsilon})"] & \mathbf{R}\mathcal{U}_{2\epsilon}(S,g)
\end{tikzcd}
\end{equation}
By Theorem 3.20 in \cite{desilva}, the functor $\mathbf{R}$ is one part of an equivalence between the categories of concrete Reeb graphs and categorical Reeb graphs. If we apply the inverse functor $\mathbf{R}^{-1}$ (the display locale functor) to \cref{set_interleaving_2} we obtain the following $\epsilon$-interleaving of Reeb graphs:
\begin{equation} \label{reeb_interleaving}
\begin{tikzcd}[column sep=huge,row sep=huge]
(R,f) \arrow[r,"\iota"] \arrow[dr,swap,"\varphi"{xshift=-12pt,yshift=10pt}] &[-20pt] \mathcal{U}_\epsilon(R,f) \arrow[r,"\iota_{\epsilon}"] \arrow[dr,swap,"\varphi_\epsilon"{xshift=-12pt,yshift=10pt}] &[-20pt] \mathcal{U}_{2\epsilon}(R,f) \\
(S,g) \arrow[r,"\iota"] \arrow[ur,crossing over,swap,"\psi"{xshift=12pt,yshift=10pt}] & \mathcal{U}_\epsilon(S,g) \arrow[ur,swap,crossing over,"\psi_\epsilon"{xshift=12pt,yshift=10pt}] \arrow[r,"\iota_{\epsilon}"] & \mathcal{U}_{2\epsilon}(S,g)
\end{tikzcd}
\end{equation}
Note that by the proof of Theorem 3.20 in \cite{desilva} and the following discussion the functors $\mathbf{R}$ and $\mathbf{R}^{-1}$ are actually inverse to each other, i.e.\ $\mathbf{R}\circ\mathbf{R}^{-1}=\text{id}$ and $\mathbf{R}^{-1}\circ \mathbf{R}=\text{id}$. In particular, we have that $\mathbf{R}(\varphi)=\mathbf{R}\circ\mathbf{R}^{-1}\big(h(\alpha)\big)=h(\alpha)$, i.e.\ , for all $U$, we obtain the following commutative diagram: 
\begin{equation} \label{alphaphi}
\begin{tikzcd}[column sep=huge,row sep=huge]
\pi_0(f^{-1}(U)) \arrow[r,"\pi_0(\varphi)"] \arrow[d,"="] & \pi_0(g_\epsilon^{-1}(U)) \arrow[d,"="] \\
\pi_0(f^{-1}(U)) \arrow[r,"h(\alpha)"] \arrow[d,"="] & \pi_0(g_\epsilon^{-1}(U))  \arrow[d,"h(U)^{-1}"]  \\
\pi_0(f^{-1}(U)) \arrow[r,"\alpha"]  & \pi_0(g^{-1}(U^\epsilon))
\end{tikzcd}
\end{equation}
using the inverse of the isomorphism $h(U)$ in \cref{hiso}. By the proof of Lemma 15 in \cite{bauer}, there exist $\Phi\colon R\rightarrow S$ and $\Psi\colon S\rightarrow R$ sucht that 
\begin{equation} \label{func_dist}
\begin{aligned}
\underset{ \begin{subarray}{c} (r,r'),(s,s') \\ \in C(\Phi,\Psi) \end{subarray} }{\text{sup}}\frac{1}{2}|d_f(r,r')-d_g(s,s')| &\leq 3(\epsilon+\delta) \\
\lvert\lvert f-g\circ\Phi \rvert\rvert_\infty &\leq \epsilon+\delta \\
\lvert\lvert g-f\circ\Psi \rvert\rvert_\infty &\leq \epsilon+\delta
\end{aligned}
\end{equation}
for all sufficiently small $\delta>0$. For $r\in R$, we now show that $BF(r)$ is close to $BG\circ\Phi(r)$ in the interleaving distance. 

Let $\kappa>0$, $t\in \mathbb{R}_{\geq 0}$ and $B(f(r),t)\subseteq \mathbb{R}$ be an open ball of radius $t$ around $f(r)$. Since $|f(r)-g\circ\Phi(r)|\leq \epsilon+\delta$, if $\kappa>\epsilon+\delta$, we get:
\begin{equation} \label{inclusions1}
\begin{aligned}
B\big(f(r),t\big)&\subseteq B\big(g\circ \Phi(r),t+\kappa\big) \\ &\subseteq B\big(g\circ \Phi(r),t+\kappa+2\epsilon\big) \\ &\subseteq B\big(f(r),t+2(\kappa+\epsilon)\big) \hspace{1pt}.
\end{aligned}
\end{equation}
Therefore, by functoriality of $\mathcal{F}$ and the $\epsilon$-interleaving between $\mathcal{F}$ and $\mathcal{G}$ in \cref{cat_interleaving} we obtain:\small
\begin{equation} \label{interleaving2step}
\begin{tikzcd}[column sep=tiny]
\mathcal{F}\Big(B\big(f(r),t\big)\Big) \arrow[rr,"\mathcal{F}(\iota)"] \arrow[d,swap,"\mathcal{F}(\iota)"] &[-50pt] &[-50pt] \mathcal{F}\Big(B\big(f(r),t+2(\kappa+\epsilon)\big)\Big) \\
\mathcal{F}\Big(B\big(g\circ\Phi(r),t+\kappa\big)\Big)\arrow[rr,"\mathcal{F}(\iota)"] \arrow[dr,swap,"(\alpha_{t+\kappa}\text{,}\eta_{t+\kappa})"{yshift=2pt}] & & \Big(B\big(g\circ\Phi(r),t+\kappa+2\epsilon\big)\Big) \arrow[u,swap,"\mathcal{F}(\iota)"] \\[5pt]
&  \mathcal{G}\Big(B\big(g\circ\Phi(r),t+\kappa+\epsilon\big)\Big) \arrow[ur,swap,"(\beta_{t+\kappa+\epsilon}\text{,}\rho_{t+\kappa+\epsilon})"{yshift=2pt}]
\end{tikzcd}
\end{equation}
\normalsize
If we apply $\mathbf{dom}$ to \cref{interleaving2step} we obtain:
\small
\begin{equation} \label{interleaiving_comp}
\begin{tikzcd}[column sep=tiny]
\pi_0\Big(f^{-1}\big(B(f(r),t)\big)\Big) \arrow[drr,"\pi_0(\iota)"] \arrow[dd,swap,"\pi_0(\iota)"] \\ &[-30pt] &[-30pt] \pi_0\Big(f^{-1}\big(B(f(r),t+2(\kappa+\epsilon))\big)\Big) \\
\pi_0\Big(f^{-1}\big(B(g\circ\Phi(r),t+\kappa)\big)\Big) \arrow[drr,"\pi_0(\iota)"] \arrow[dd,swap,"\alpha_{t+\kappa}"] \\ & & \pi_0\Big(f^{-1}\big(B(g\circ\Phi(r),t+\kappa+2\epsilon)\big)\Big) \arrow[uu,swap,"\pi_0(\iota)"] \\
\pi_0\Big(g^{-1}\big(B(g\circ\Phi(r),t+\kappa+\epsilon)\big)\Big) \arrow[urr,swap,"\beta_{t+\kappa+\epsilon}"]
\end{tikzcd}
\end{equation}
\normalsize
Let $B_{d_f}(r,t)$ be the open ball of radius $t$ around $r$ in $R$. Since, $B_{d_f}(r,t)\subseteq f^{-1}(B(f(r),t))$ is  by definition path-connected, $B_{d_f}(r,t)\in \pi_0\Big(f^{-1}\big(B(f(r),t)\big)\Big)$ and, since $r\in B_{d_f}(r,t)$, we have $B_{d_f}(r,t)=[r]$ the path-component of $r$ in $f^{-1}(B(f(r),t))$. By the same argument,
$B_{d_g}(\Phi(r),t+\kappa+\epsilon)=[\Phi(r)]\in \pi_0\Big(g^{-1}\big(B(g\circ\Phi(r),t+\kappa+\epsilon)\big)\Big)$.  Moreover, $\pi_0(\iota)([r])=[\iota(r)]=[r]\in \pi_0\Big(f^{-1}\big(B(g\circ\Phi(r),t+\kappa)\big)\Big) $. By using \cref{alphaphi} for $U=B(g\circ\Phi(r),t+\kappa)$ we obtain:
\begin{equation}
\begin{tikzcd}
\pi_0\big(f^{-1}(B(g\circ\Phi(r),t+\kappa))\big) \arrow[dr,"\pi_0(\varphi)"] \arrow[dd,swap,"\alpha"] \\[5pt] &[-60pt] \pi_0\big(g_{\epsilon}^{-1}(B(g\circ\Phi(r),t+\kappa))\big) \arrow[dl,"h^{-1}"] \\[5pt] 
\pi_0\big(g^{-1}(B(g\circ\Phi(r),t+\kappa+\epsilon))\big)
\end{tikzcd}  
\end{equation}
\normalsize
By \cref{hiso}, $h^{-1}\coloneqq \pi_0(p_1)\circ\pi_0(q)^{-1}$ and, by \cite{bauer} Section 3.2, $\Phi\coloneqq p_1\circ \tilde{\varphi}_\delta$.
Since $\tilde{\varphi}_\delta(r)\in \overline{\varphi}(B_{d_f}(r,\delta))=q^{-1}\big(\varphi(B_{d_f}(r,\delta))\big)$, $\varphi(B_{d_f}(r,\delta))$ is path-connected and $\varphi(r)\in \varphi(B_{d_f}(r,\delta))$, we have that $[\varphi(r)]=[q(\tilde{\varphi}_\delta(r))]=\pi_0(q)([\tilde{\varphi}_\delta(r)])$. Hence, $\pi_0(q)^{-1}([\varphi(r)])=[\tilde{\varphi}_\delta(r)]$. By definition of $\Phi$, we have $[\Phi(r)]=[p_1\circ\tilde{\varphi}_\delta(r)]=\pi_0(p_1)([\tilde{\varphi}_\delta(r)])$. Therefore, $h^{-1}\circ \pi_0(\varphi)([r])=h^{-1}([\varphi(r)])=\pi_0(p_1)\circ\pi_0(q)^{-1}([\varphi(r)])=\pi_0(p_1)([\tilde{\varphi}_\delta(r)])=[\Phi(r)]=\alpha_{t+\kappa}([r])$. By commutativity of \cref{interleaiving_comp}, $\beta_{t+\kappa+\epsilon}\circ \alpha_{t+\kappa}([r])=\beta_{t+\kappa+\epsilon}([\Phi(r)])=\pi_0(\iota)([r])=[r]$. As a consequence, $\alpha_{t+\kappa}\circ \pi_0(\iota)(B_{d_f}(r,t))=B_{d_g}(\Phi(r),t+\kappa+\epsilon)$ and $\pi_0(\iota)\circ \beta_{t+\kappa+\epsilon}(B_{d_g}(\Phi(r),t+\kappa+\epsilon))=B_{d_f}(r,t+2(\kappa+\epsilon))$. Thus, the interleaving in \cref{interleaving2step} yields the following commutative diagram in $\mathbf{Vec}$:
\begin{equation} \label{barcode_interleaving1}
\begin{tikzcd}[column sep=small]
F\big(B_{d_f}(r,t)\big) \arrow[rr,"F(\iota)"] \arrow[d,swap,"F(\iota)"] &[-40pt] &[-40pt] F\big(B_{d_f}(r,t+2(\kappa+\epsilon))\big) \\ 
F([r]) \arrow[dr,swap,"\eta_{t+\kappa}(\text{[r]})"{yshift=2pt}] \arrow[rr,"F(\iota)"] & & F([r]) \arrow[u,swap,"F(\iota)"] \\[5pt]
& G\big(B_{d_g}(\Phi(r),t+\kappa+\epsilon)\big) \arrow[ur,swap,"\rho_{t+\kappa+\epsilon}(\text{[}\Phi(r)\text{]})"{yshift=2pt}]
\end{tikzcd}
\end{equation}
\normalsize
where $[r]$ and $[\Phi(r)]$ denote the topological path-components in the respective preimages. 

We now start with $\Phi(r)$. Similar to \cref{inclusions1} we obtain the following inclusions of open intervals in $\mathbb{R}$:
\begin{equation} \label{inclusions2}
\begin{aligned}
B\big(g\circ \Phi(r),t\big)&\subseteq B\big(f(r),t+\kappa\big) \\ &\subseteq B\big(f(r),t+\kappa+2\epsilon\big) \\ &\subseteq B\big(g\circ \Phi(r),t+2(\kappa+\epsilon)\big)
\end{aligned}
\end{equation}
for every $\kappa>\epsilon+\delta$. Therefore, by functoriality of $\mathcal{G}$ and the $\epsilon$-interleaving between $\mathcal{F}$ and $\mathcal{G}$ in \cref{cat_interleaving} we obtain:
\small
\begin{equation} \label{interleaving2step2}
\begin{tikzcd}[column sep=tiny]
\mathcal{G}\Big(B\big(g\circ\Phi(r),t\big)\Big) \arrow[rr,"\mathcal{G}(\iota)"] \arrow[d,swap,"\mathcal{G}(\iota)"] &[-40pt] &[-40pt] \mathcal{G}\Big(B\big(g\circ\Phi(r),t+2(\kappa+\epsilon)\big)\Big) \\
\mathcal{G}\Big(B\big(f(r),t+\kappa\big)\Big)\arrow[rr,"\mathcal{G}(\iota)"] \arrow[dr,swap,"(\beta_{t+\kappa}\text{,}\rho_{t+\kappa})"{yshift=2pt}] & & \mathcal{G}\Big(B\big(f(r),t+\kappa+2\epsilon\big)\Big) \arrow[u,swap,"\mathcal{G}(\iota)"] \\[5pt]
& \mathcal{F}\Big(B\big(f(r),t+\kappa+\epsilon\big)\Big) \arrow[ur,swap,"(\alpha_{t+\kappa+\epsilon}\text{,}\eta_{t+\kappa+\epsilon})"{yshift=2pt}]
\end{tikzcd}
\end{equation}
\normalsize and, by applying $\mathbf{dom}$, we get: 
\small
\begin{equation} \label{interleaiving_comp2}
\begin{tikzcd}[column sep=tiny]
\pi_0\Big(g^{-1}\big(B(g\circ\Phi(r),t)\big)\Big) \arrow[drr,"\pi_0(\iota)"] \arrow[dd,swap,"\pi_0(\iota)"] \\ &[-30pt] &[-30pt] \pi_0\Big(g^{-1}\big(B(g\circ\Phi(r),t+2(\kappa+\epsilon))\big)\Big) \\
\pi_0\Big(g^{-1}\big(B(f(r),t+\kappa)\big)\Big) \arrow[drr,"\pi_0(\iota)"] \arrow[dd,swap,"\beta_{t+\kappa}"] \\ & & \pi_0\Big(g^{-1}\big(B(f(r),t+\kappa+2\epsilon)\big)\Big) \arrow[uu,swap,"\pi_0(\iota)"] \\
\pi_0\Big(f^{-1}\big(B(f(r),t+\kappa+\epsilon)\big)\Big) \arrow[urr,swap,"\alpha_{t+\kappa+\epsilon}"]
\end{tikzcd}
\end{equation}
\normalsize
As in the previous case, we have that $B_{d_g}(\Phi(r),t)\subseteq g^{-1}\big(B(g\circ\Phi(r),t)\big)$ is the path-component of $\Phi(r)$, i.e.\ $[\Phi(r)]=B_{d_g}(\Phi(r),t)\in\pi_0\Big(g^{-1}\big(B(g\circ\Phi(r),t)\big)\Big)$ and, analogously, $[r]=B_{d_f}(r,t+\kappa+\epsilon)\in\pi_0\Big(f^{-1}\big(B(f(r),t+\kappa+\epsilon)\big)\Big)$. We now use the analog of \cref{alphaphi} for $U=B(f(r),t+\kappa)$, $\psi$ from the interleaving in \cref{reeb_interleaving} and $\beta$ to obtain:
\begin{equation}
\begin{tikzcd}
\pi_0\big(g^{-1}(B(f(r),t+\kappa))\big) \arrow[dr,"\pi_0(\psi)"] \arrow[dd,swap,"\beta"] \\[5pt] &[-30pt] \pi_0\big(f_{\epsilon}^{-1}(B(f(r),t+\kappa))\big) \arrow[dl,"h^{-1}"] \\[5pt] 
\pi_0\big(f^{-1}(B(f(r),t+\kappa+\epsilon))\big)
\end{tikzcd}  
\end{equation}
\normalsize
By \cref{hiso}, $h^{-1}\coloneqq \pi_0(p_1)\circ\pi_0(q)^{-1}$ and, by \cite{bauer} Section 3.2, $\Psi\coloneqq p_1\circ \tilde{\psi}_\delta$.
Since $\tilde{\psi}_\delta(\Phi(r))\in \overline{\psi}(B_{d_g}(\Phi(r),\delta))=q^{-1}\big(\psi(B_{d_g}(\Phi(r),\delta))\big)$, $\psi(B_{d_g}(\Phi(r),\delta))$ is path-connected and $\psi(\Phi(r))\in \psi(B_{d_g}(\Phi(r),\delta))$, we have that $[\psi(\Phi(r))]=[q(\tilde{\psi}_\delta(\Phi(r)))]=\pi_0(q)([\tilde{\psi}_\delta(\Phi(r))])$. Hence, $\pi_0(q)^{-1}([\psi(\Phi(r))])=[\tilde{\psi}_\delta(\Phi(r))]$. By definition of $\Psi$, we have $[\Psi(\Phi(r))]=[p_1\circ\tilde{\psi}_\delta(\Phi(r))]=\pi_0(p_1)([\tilde{\psi}_\delta(\Phi(r))])$. Therefore, $h^{-1}\circ \pi_0(\psi)([\Phi(r)])=h^{-1}([\psi(\Phi(r))])=\pi_0(p_1)\circ\pi_0(q)^{-1}([\psi(\Phi(r))])=\pi_0(p_1)([\tilde{\psi}_\delta(\Phi(r))])=[\Psi(\Phi(r))]=\beta_{t+\kappa}([\Phi(r)])$.

From \cref{func_dist} we get $\frac{1}{2}|d_f(r,\Psi\circ \Phi(r))|\leq 3(\epsilon+\delta)$. If $\kappa+\epsilon>6(\epsilon+\delta)$, then $B_{d_f}(r,6(\epsilon+\delta))\subseteq B_{d_f}(r,t+\kappa+\epsilon)\subseteq f^{-1}\big(B(f(r),t+\kappa+\epsilon)\big)$. Hence, since  $r$ and $\Psi\circ \Phi(r)\in B_{d_f}(r,t+\kappa+\epsilon)$ and $B_{d_f}(r,t+\kappa+\epsilon)$ is path-connected, $[r]=[\Psi\circ \Phi(r)]\in \pi_0\big(f^{-1}(B(f(r),t+\kappa+\epsilon))\big)$.

Therefore, starting with $B_{d_g}(\Phi(r),t)=[r]$, we obtain $\beta_{t+\kappa}\circ\pi_0(\iota)([\Phi(r)])=\beta_{t+\kappa}([\Phi(r)])=[\Psi\circ\Phi(r)]=[r]=B_{d_f}(r,t+\kappa+\epsilon)$. This implies that we can extract the following commutative diagram from \cref{interleaving2step2}:
\begin{equation} \label{barcode_interleaving2}
\begin{tikzcd}[column sep=tiny]
G\big(B_{d_g}(\Phi(r),t)\big) \arrow[rr,"G(\iota)"] \arrow[d,swap,"G(\iota)"] &[-40pt] &[-40pt] G\big(B_{d_g}(\Phi(r),t+2(\kappa+\epsilon))\big) \\ 
G([\Phi(r)]) \arrow[dr,swap,"\rho_{t+\kappa}(\text{[}\Phi(r)\text{]})"] \arrow[rr,"G(\iota)"] & & G([\Phi(r)]) \arrow[u,swap,"G(\iota)"] \\[5pt]
& F\big(B_{d_f}(r,t+\kappa+\epsilon)\big) \arrow[ur,swap,"\eta_{t+\kappa+\epsilon}(\text{[}r\text{]})"]
\end{tikzcd}
\end{equation}
Now we define 
\begin{equation}
\begin{aligned}
& \mu_t\colon F\big(B_{d_f}(r,t)\big)\rightarrow G\big(B_{d_g}(\Phi(r),t+\kappa+\epsilon)\big) \\
& \mu_t\coloneqq \eta_{t+\kappa}\circ F(\iota) \\
& \nu_t\colon G\big(B_{d_g}(\Phi(r),t)\big)\rightarrow F\big(B_{d_f}(r,t+\kappa+\epsilon)\big) \\
& \nu_t\coloneqq \rho_{t+\kappa}\circ G(\iota)
\end{aligned}
\end{equation}
Since $\mathcal{F}$ and $\mathcal{G}$ are $\epsilon$-interleaved we have the following commutative diagram
\begin{equation}
\begin{tikzcd}
\mathcal{F}\big(B(f(r),t)\big) \arrow[r,"\mathcal{F}(\iota)"] \arrow[d,swap,"(\alpha_{t}\text{,}\eta_{t})"] & \mathcal{F}\big(B(g\circ\Phi(r),t+\kappa)\big) \arrow[d,"(\alpha_{t+\kappa}\text{,}\eta_{t+\kappa})"] \\[10pt]
\mathcal{G}\big(B(f(r),t+\epsilon)\big) \arrow[r,"\mathcal{G}(\iota)"] & \mathcal{G}\big(B(g\circ\Phi(r),t+\kappa+\epsilon)\big)
\end{tikzcd}
\end{equation}
Following the component $B_{d_f}(r,t)$ we get
\begin{equation}
\begin{tikzcd}
F\big(B_{d_f}(r,t)\big) \arrow[r,"F(\iota)"] \arrow[d,swap,"\eta_{t}"] &[20pt] F\big([r]\big) \arrow[d,"\eta_{t+\kappa}"] \\[10pt]
G\big(\ldots\big) \arrow[r,"\mathcal{G}(\iota)"] & G\big(B_{d_g}(\Phi(r),t+\kappa+\epsilon)\big)
\end{tikzcd}
\end{equation}
This implies that the map $\mu_{t}=\eta_{t+\kappa}\circ F(\iota)=G(\iota)\circ \eta_{t}$. Analogously we obtain that $\nu_{t}=\rho_{t+\kappa}\circ G(\iota)=F(\iota)\circ\rho_{t}$. Moreover, for $s<t\in\mathbb{R}_{\geq 0}$, the following diagram and its analog for $\nu$ obviously commute:
\small
\begin{equation}
\begin{tikzcd}
F\big(B_{d_f}(r,s)\big) \arrow[r,"F(\iota)"] \arrow[d,swap,"\mu_s"] & F\big(B_{d_f}(r,t)\big) \arrow[d,"\mu_t"] \\
G\big(B_{d_g}(\Phi(r),s+\kappa+\epsilon)\big) \arrow[r,"G(\iota)"] & G\big(B_{d_g}(\Phi(r),t+\kappa+\epsilon)\big)
\end{tikzcd}
\end{equation}
\normalsize
Combining these results with \cref{barcode_interleaving1} and \cref{barcode_interleaving2}, we obtain the following $(\kappa+\epsilon)$-interleaving:
\small
\begin{equation}
\begin{tikzcd}[row sep=large]
F\big(B_{d_f}(r,t)\big) \arrow[d,swap,"F(\iota)"] \arrow[dr,"\mu_t"{xshift=20pt,yshift=-6pt}] & G\big(B_{d_g}(\Phi(r),t)\big) \arrow[d,"G(\iota)"] \arrow[dl,"\nu_t"{xshift=-25pt,yshift=4pt},crossing over] \\
F\big(B_{d_f}(r,t+\kappa+\epsilon)\big) \arrow[d,swap,"F(\iota)"] \arrow[dr,"\mu_{t+\kappa+\epsilon}"{xshift=20pt,yshift=-6pt}] & G\big(B_{d_g}(\Phi(r),t+\kappa+\epsilon)\big) \arrow[d,"G(\iota)"] \arrow[dl,"\nu_{t+\kappa+\epsilon}"{xshift=-39pt,yshift=4pt},crossing over] \\
F\big(B_{d_f}(r,t+2(\kappa+\epsilon))\big) & G\big(B_{d_g}(\Phi(r),t+2(\kappa+\epsilon))\big) 
\end{tikzcd}
\end{equation}
\normalsize
Hence, $BF(r)$ and $BG(\Phi(r))$ are $(\kappa+\epsilon)$-interleaved for every $\kappa>5\epsilon+6\delta$. Since $\text{inf}\{\kappa+\epsilon\mid \kappa>5\epsilon+6\delta\text{ and }\delta>0\}=6\epsilon$, we finally obtain $d_\mathrm{I}\big(BF(r),BG(\Phi(r))\big)\leq 6\epsilon$. By symmetry, we analogously obtain $d_\mathrm{I}\big(BF(\Psi(s)),BG(s)\big)\leq 6\epsilon$. Together with \cref{func_dist}, these bounds imply the theorem.

\end{document}